\documentclass[review,onefignum,onetabnum]{siamart190516}

\usepackage{lipsum}
\usepackage{amsfonts}
\usepackage{mathtools}
\usepackage{graphicx}
\usepackage{tikz}
\usepackage{epstopdf}
\usepackage{algorithmic}
\ifpdf
  \DeclareGraphicsExtensions{.eps,.pdf,.png,.jpg}
\else
  \DeclareGraphicsExtensions{.eps}
\fi

\usepackage{xcolor} 
\usepackage{graphicx}
\usepackage{subcaption}
\graphicspath{{Pictures/}}

\newcommand{\pvec}{\mathbf{p}}

\newcommand{\nvec}{\mathbf{n}}
\newcommand{\Qvec}{\mathbf{Q}}

\newcommand{\Mvec}{\mathbf{M}}
\newcommand{\mvec}{\mathbf{m}}

\allowdisplaybreaks 

\newsiamremark{remark}{Remark}
\newsiamremark{hypothesis}{Hypothesis}
\crefname{hypothesis}{Hypothesis}{Hypotheses}
\newsiamthm{claim}{Claim}

\headers{One-dimensional ferronematics in a channel}{J. Dalby, P. E. Farrell, A. Majumdar and J. Xia}

\title{One-dimensional ferronematics in a channel: order reconstruction, bifurcations and multistability\thanks{Submitted to the editors DATE.
\funding{
PEF is supported by the Engineering and Physical Sciences Research Council [grant numbers EP/R029423/1 and EP/V001493/1].
AM and JD are supported by a DST-UKIERI grant on ``Theoretical and experimental studies of suspensions of magnetic nanoparticles, their applications and generalisations''. AM is supported by a Leverhulme International Academic Fellowship, the University of Strathclyde's New Professor Fund and an OCIAM Visiting Fellowship. AM thanks Giacomo Canevari for informative discussions on $\Gamma$-convergence and Neela Nataraj, Ruma Maity for discussions on numerical analysis. AM thanks Varsha Banerjee and Konark Bisht for their collaboration in 2019.
    JX is supported by the EPSRC Centre for Doctoral Training in Partial Differential Equations [grant number EP/L015811/1] and the National University of Defense Technology.
}}}

\author{James Dalby\thanks{Department of Mathematics, University of Strathclyde, UK
  (\email{james.dalby@strath.ac.uk}, \email{apala.majumdar@strath.ac.uk}).}
\and Patrick E. Farrell\thanks{Mathematical Institute, University of Oxford, UK 
(\email{patrick.farrell@maths.ox.ac.uk}).}
\and Apala Majumdar\footnotemark[2]
\and Jingmin Xia \thanks{College of Meteorology and Oceanography, National University of Defense Technology, China
  (\email{jingmin.xia@nudt.edu.cn}).}
}

\usepackage{amsopn}

\begin{document}

\maketitle

\begin{abstract}
    We study a model system with nematic and magnetic order, within a channel geometry modelled by an interval, $[-D, D]$. The system is characterised by a tensor-valued nematic order parameter $\Qvec$ and a vector-valued magnetisation $\Mvec$, and the observable states are modelled as stable critical points of an appropriately defined free energy which includes a nemato-magnetic coupling term, characterised by a parameter $c$.
    We (i) derive $L^\infty$ bounds for $\Qvec$ and $\Mvec$; (ii) prove a uniqueness result in specified parameter regimes; 
    (iii) analyse order reconstruction solutions, possessing domain walls, and their stabilities as a function of $D$ and $c$ and (iv) perform numerical studies that elucidate the interplay of $c$ and $D$ for multistability.
\end{abstract}



\section{Introduction}
\label{sec:intro}
Nematic liquid crystals (NLCs) are classical examples of meso-phases that combine fluidity with long-range orientational order \cite{deGennes}.
NLC molecules tend to align, on average, along certain locally preferred directions, referred to as nematic \emph{directors}.
NLCs are anisotropic materials with a direction-dependent response to light and external fields, and are thus used in a range of electro-optical devices, e.g., the multi-billion dollar liquid crystal display industry \cite{lagerwallreview}.
Moreover, NLCs typically rely on their dielectric anisotropy, i.e., directional response to external electric fields, for applications.
Their responses to external magnetic fields are much weaker (perhaps seven orders of magnitude smaller) than their dielectric response \cite{stewart} and consequently, nemato-magnetic coupling has been poorly exploited for NLC applications, e.g., sensors, displays, microfluidics etc.

In the pioneering work of \cite{brochardgennes}, Brochard and de Gennes suggested that a suspension of magnetic nanoparticles (MNPs) in a NLC host could induce a spontaneous magnetisation without any external magnetic fields, and substantially enhance nemato-magnetic material response.
This new class of materials with both nematic and magnetic order is referred to as \emph{ferronematics}, with notable theoretical contributions by \cite{burylov, calderer2014} and experimental realisations by \cite{cladis}, later by \cite{mertelj-2013-article} where the crucial factors for the stability of ferronematic suspensions are identified.
Ferronematics have tremendous potential, both theoretically and for meta-materials, topological materials, and nano-systems, to name a few \cite{smalyukh}.
Of particular interest are multistable ferronematic systems that support multiple stable ferronematic states, without external magnetic fields. This is analogous to multistable nematic systems, such as bistable liquid crystal displays, but ferronematics have additional magnetic order that allows for greater complexity of solution landscapes. This work is a first step in the rigorous analytical and numerical study of multistable one-dimensional ferronematic systems, without external magnetic fields. Magnetic fields could be used to switch between the distinct stable ferronematic states, to control non-equilibrium behaviour for such multistable systems.

In this work, we study a dilute suspension of MNPs in a one-dimensional NLC-filled channel (of width $D$). 
We assume a uniform distribution of MNPs (much smaller than the physical domain dimensions) such that the average distance between the MNPs is much larger than the MNP size, and the total volume fraction of MNPs is small.
These MNPs generate a spontaneous magnetisation even without any external magnetic fields, by means of the NLC-MNP interactions. Thus, the system has two order parameters: (i) a reduced Landau--de Gennes (LdG) nematic tensor parameter $\Qvec$ with two degrees of freedom, that contains information about the nematic directors and the degree of nematic ordering and (ii) a magnetisation vector $\Mvec$ generated by the suspended MNPs.

Following the methods in \cite{bisht-2019-article, bisht-2020-article, calderer2014}, we model the physically observable $(\Qvec, \Mvec)$-profiles as minimisers of an appropriately defined ferronematic free energy.
This free energy consists of three contributions: a LdG-type nematic energy, a magnetisation energy and a nemato-magnetic coupling energy.
In fact, the free energy essentially builds on the energy in \cite{burylov}, with two differences: we describe the nematic state by a LdG-type order parameter instead of a unit-vector as in \cite{burylov}, and we add the magnetisation energy to essentially regularise the problem, i.e., the magnetisation energy penalises sharp jumps or inhomogeneities in $\Mvec$.
The LdG tensor order parameter is well suited to capture fractional point defects as in \cite{bisht-2020-article} and biaxiality in three dimensions, i.e., primary and secondary nematic directors which are outside the scope of a purely vector-based model as in \cite{burylov}. Further, as shown in \cite{calderer2014} and \cite{canevari2020design}, in the dilute limit, the microscopic details of the MNP properties (shape, size, anchoring on the MNP surfaces, volume fraction etc.) and the NLC-MNP interactions are homogenised to yield the nemato-magnetic coupling energy, characterised by a coupling parameter $c > 0$. The coupling energy dictates the co-alignment between the nematic director and $\Mvec$ and for positive $c$ as in our manuscript, this coupling energy favours that the director and $\Mvec$ be parallel to each other.
There are four key phenomenological parameters in the ferronematic free energy as in \cite{bisht-2020-article}: $l_1$ and $l_2$ which depend on elastic constants, the temperature and are inversely proportional to $D^2$; the nemato-coupling parameter $c$; and a scaling parameter $\xi$ that weighs the relative strength of the nematic and magnetic energies. For dilute systems, $\xi$ is typically small.
In addition, we prescribe conflicting Dirichlet conditions for $\Qvec$ and $\Mvec$, that necessarily generate inhomogeneous ferronematic profiles. The physically relevant choices of the boundary conditions for $\Mvec$ are unclear, but we expect our conclusions to be qualitatively unchanged with Neumann boundary conditions for $\Mvec$.


The pure nematic case, i.e., when $c=0$, is well-understood; see for example \cite{canevari_majumdar_spicer2016, lamy-2014-article}.
We study how the solution landscapes for $c=0$ are perturbed by the nemato-magnetic coupling energy in the dilute limit, for positive $c$.
In the supplementary material, we compute the \emph{vacuum manifold} i.e., minimisers of the bulk potential, which is the sum of the Ginzburg--Landau energies for $\Qvec$ and $\Mvec$ and a nemato-magnetic coupling energy, and the bulk minimisers depend on $c$ and $\xi$. The bulk minimisers are the spatially homogeneous profiles that would be observed without conflicting boundary conditions or geometrical frustration and they play a crucial role in our study of this one-dimensional spatially inhomogeneous problem of ferronematics in channel geometries. We next prove the existence of minimisers of the ferronematic free energy (\cref{thm:existance}) for this model problem, subject to the conflicting Dirichlet boundary conditions for $\Qvec$ and $\Mvec$. The minimisers (local and global) are candidates for physically observable configurations. We then prove a non-trivial  maximum principle (\cref{thm:maximum-principle}) for all critical points of the ferronematic free energy, and we obtain an explicit $L^\infty$ bound for the critical points $(\Qvec, \Mvec)$ in terms of $c$. This bound strongly depends on our analysis of the vacuum manifold. In particular, this bound reduces to the familiar uncoupled bound for $c=0$ in \cite{canevari_majumdar_spicer2016}, with a linear perturbation in $c$ for small $c$. For large $c$, the bounds grow linearly with $c$. This captures the relationship between the $c=0$ and $c>0$ cases to some extent. 
Subsequently, in \cref{thm:uniqness} we prove that the ferronematic energy has a unique critical point, and hence minimiser, for $D$ sufficiently small, i.e., for narrow channels, as for the $c=0$ case in \cite{lamy-2014-article}. Of course, the critical $D$ depends on $c$. These crucial analytic results hold in two and three dimensions too, and are hence of general interest.

In the pure nematic case ($c = 0$), the model problem admits a unique order reconstruction (OR) solution for $D \ll c_2 \xi_n$, for some positive constant $c_2$ independent of model parameters, and where $\xi_n$ is the nematic correlation length \cite{lamy-2014-article}. OR solutions are special since they support polydomains, separated by domain walls, such that the nematic director is constant in each polydomain and jumps across the domain wall. These polydomains are stable for $D$ small enough, and become unstable as $D$ increases.
The qualitative features are unchanged in the ferronematic case, where profiles have four degrees of freedom: two for $\Qvec$ and two for $\Mvec$. Here, an OR solution exists for all $D$ (\cref{thm:existence_OR}), with distinct domain walls (defined by $\Qvec= 0$ and $\Mvec = 0$) that separate distinctly ordered polydomains for both the nematic director and the magnetisation vector. The OR solutions are reduced solutions with only two degrees of freedom and the polydomains are a necessary consequence of the Dirichlet boundary conditions.
Essentially, the polydomains have a constant non-zero $(\Qvec, \Mvec)$-profile and the profile jumps across a domain wall, which is the surface discontinuity in the three dimensional channel setting.
Moreover, OR solutions are globally stable for $D \ll c_1 \frac{\xi_n}{\sqrt{c_0^2 + c^2}}$ and thus $c$ shrinks their domain of stability.
In \cref{thm:instability-OR}, as $D$ increases, we show that OR solutions become unstable by means of a $\Gamma$-convergence argument and second variation analysis.
Next, we study the full problem with four degrees of freedom. As $D$ increases, the ferronematic energy minimisers lose the polydomain structures and the nematic director and the magnetisation vector rotate smoothly throughout the channel. For large $D$, these minimisers exploit the full four degrees of freedom, and we have boundary layers because the boundary conditions are not consistent with the vacuum manifold. This is further corroborated by numerical experiments and computations of bifurcation diagrams, for two specific values of $c$. As $D$ increases, we observe pitchfork bifurcations from the OR solutions and multiple stable ferronematic equilibria for large $D$, demonstrating an example of a multistable ferronematic system.

The nemato-magnetic coupling introduces additional possibilities for the interplay between nematic and magnetic domain walls (absent when $c=0$), new defect structures, and novel bifurcations accompanied by novel solution branches for large $D$. In particular, the ferronematic OR solutions illustrate how we can tailor the locations and multiplicity of domain walls by varying $D$ and $c$, a novel aspect of our study. We do not address these questions fully in this manuscript but our work will support and guide future studies on these lines. The paper is organised as follows. In the next section, we describe the ferronematic model and give some qualitative results of general interest, e.g., existence, uniqueness 
etc.
We then consider the OR model in \cref{sec:OR} and provide numerical results in \cref{sec:num} to verify our theoretical analysis. Finally, some conclusions and perspectives are summarised in \cref{sec:conclusions}.

\section{Model Problem}
\label{sec:fullproblem}

We consider a dilute ferronematic suspension sandwiched inside the three dimensional channel $\tilde{\Omega} = [-L, L] \times [-D, D] \times [0, G]$, where $L \gg D$, $L$ is the length of the channel, $D$ is the channel width and $G$ is the channel height. We impose strong anchoring on the $xz$-planes and free boundary conditions on the $yz$- and $xy$-planes.
From a modelling perspective, we assume that the structural profile is invariant across the height of the channel, and along the length of the channel and restrict ourselves to a one-dimensional channel geometry: $\Omega = [-D, D]$ in what follows.
As noted from \cref{sec:intro}, the ferronematic suspension is described by two order parameters: a symmetric, traceless $2\times 2$ matrix $\Qvec$, i.e., $\Qvec\in S_0\coloneqq\{\Qvec\in\mathbb{M}^{2\times 2}: Q_{ij}=Q_{ji},Q_{ii}=0\}$, and a two-dimensional vector, $\Mvec = \left(M_1, M_2 \right)$.
Here, $\mathbb{M}^{2\times 2}$ denotes all $2\times 2$ matrices. 
The nematic order parameter $\Qvec$ can be written as
\begin{equation}
    \label{eq:2}
    \Qvec = s (2 \nvec\otimes \nvec - \mathbf{I}),
\end{equation}
where $s$ is a scalar order parameter, and $\nvec$ is the nematic director (a unit-vector describing the direction of orientational ordering in the $xy$-plane) and $\mathbf{I}$ is the $2\times 2$ identity matrix.
Moreover, $s$ is interpreted as the degree of the orientational order about $\nvec$, so that the nodal sets of $s$ (i.e., where $s=0$) define nematic defects in the $xy$-plane.
We denote the two independent components of $\Qvec$ by $Q_{11}$ and $Q_{12}$ such that
        \begin{equation*}
            Q_{11} = s\cos 2 \vartheta,\quad Q_{12} = s \sin 2\vartheta,
        \end{equation*}
        when $\nvec = \left(\cos \vartheta, \sin \vartheta \right)$ and $\vartheta$ denotes the angle between $\nvec$ and the horizontal axis.
        To avoid writing $\Qvec$ in the matrix form $\bigl[\begin{smallmatrix} Q_{11} & Q_{12}\\ Q_{12} & -Q_{11}\end{smallmatrix}\bigr]$, we henceforth label $\Qvec$ in terms of its two independent components $(Q_{11}, Q_{12})$, when this causes no confusions. We therefore define the vector norm, $|\Qvec|=\sqrt{Q^2_{11}+Q^2_{12}}$, as opposed to a matrix norm. Similarly, we define $|\Mvec|=\sqrt{M_1^2+M_2^2}$. 

Following the methods in \cite{mertelj-2013-article, bisht-2019-article}, the ferronematic free energy is given by the sum of three energies for low temperatures: a LdG type nematic energy for $\Qvec$, a magnetisation energy for $\Mvec$ and a coupling energy between $\Qvec$ and $\Mvec$. For dilute ferronematic suspensions, the MNP interactions are ``small'' and the NLC-MNP interactions are absorbed by the coupling energy, which can be viewed as the homogenised version of a Rapini--Papoular type surface anchoring energy on the MNP surfaces that dictates the co-alignment between $\nvec$ and $\Mvec$ \cite{calderer2014, canevari2020design}. We adopt the rescalings as in \cite{bisht-2019-article}, so that the rescaled domain is $\Omega = [-1, 1]$, and the total rescaled and dimensionless ferronematic free energy is
\begin{equation}
    \label{eq:4}
    \begin{aligned}
        F(Q_{11}, Q_{12}, M_1, M_2)
    \coloneqq & \int_{\Omega} \Bigg\{\frac{l_1}{2}\left[\left( \frac{\mathrm{d} Q_{11}}{\mathrm{d} y}\right)^2 + \left(\frac{\mathrm{d} Q_{12}}{\mathrm{d} y}\right)^2 \right] + \left( Q_{11}^2 + Q_{12}^2 - 1 \right)^2 \\
                                          & + \frac{\xi l_2}{2} \left[ \left(\frac{\mathrm{d} M_1}{\mathrm{d} y}\right)^2 + \left(\frac{\mathrm{d} M_2}{\mathrm{d} y}\right)^2 \right] + \frac{\xi}{4}\left(M_1^2 + M_2^2 - 1 \right)^2 \\
                                          & - cQ_{11}\left(M_1^2 - M_2^2 \right) - 2c Q_{12}M_1 M_2 \Bigg\} ~\mathrm{d}y.
    \end{aligned}
\end{equation}
Here, $l_1>0$ and $l_2>0$ are scaled elastic constants (inversely proportional to $D^2$, i.e., the squared channel width).
Substituting \cref{eq:2} into the coupling energy, we observe that
\begin{equation*}
    \left(- cQ_{11}\left(M_1^2 - M_2^2 \right) - 2c Q_{12}M_1 M_2 \right) \propto -c \left(\nvec \cdot \Mvec \right)^2.
\end{equation*}

We only focus on positive coupling, i.e., $c>0$ in this work so that this coupling energy favours $\nvec \cdot \Mvec = \pm 1$. We further denote the bulk energy density by
\begin{equation}
    \label{eq:f}
    \begin{aligned}
        f(Q_{11},Q_{12},M_1,M_2) &\coloneqq \left( Q_{11}^2 + Q_{12}^2 - 1 \right)^2 + \frac{\xi}{4}\left(M_1^2 + M_2^2 - 1 \right)^2\\
               &\quad - cQ_{11}\left(M_1^2 - M_2^2 \right) - 2c Q_{12}M_1 M_2.
    \end{aligned}
\end{equation}

Regarding boundary conditions, we work with Dirichlet conditions for $\Qvec$ and $\Mvec$ on the boundaries $y=\pm 1$ i.e.,
\begin{equation}
\label{eq:5}
\begin{aligned}
    & Q_{11}\left(-1 \right) = M_1 \left(-1 \right) = 1, \\
    & Q_{12}(-1) = Q_{12}(1) = M_2(-1) = M_2 (1) = 0, \\
    & Q_{11}\left(1 \right) = M_1 \left(1 \right) = -1.
\end{aligned}
\end{equation}

Here, the boundary conditions for $\Qvec$ correspond to $\nvec = (1, 0)$ on $y=-1$ and $\nvec = (0,1)$ on $y=1$, hence, we have planar boundary conditions on $y=-1$ and normal/homeotropic boundary conditions on $y=+1$. Furthermore, the boundary conditions for $\Mvec$ describe a $\pi$-rotation between the bounding plates, $y=\pm 1$.

The admissible space is given by
\begin{multline}
        \label{eq:admissible_space}
     \mathcal{A} = 
    \left\{\Qvec \in W^{1,2}\left(\Omega; S_0 \right), \Mvec \in W^{1,2}\left(\Omega; \mathbb{R}^2 \right)\right.,\\
    \left.\textrm{$\Qvec$ and $\Mvec$ satisfy the boundary conditions \cref{eq:5}} \right\}.
\end{multline}
    The Sobolev space $W^{1,2}$ is the space of all square-integrable $\left(\Qvec, \Mvec \right)$ with square-integrable first weak derivatives, which is a standard choice for such variational problems.
The stable, physically relevant and potentially observable $(\Qvec, \Mvec)$-profiles are local or global energy minimisers of the full energy \cref{eq:4} subject to the boundary conditions in \cref{eq:5}, in $\mathcal{A}$.
They are in fact, classical solutions of the associated Euler--Lagrange equations \cite{bisht-2019-article}
\begin{subequations}
    \label{eq:euler-lagrange}
\begin{align}
    & l_1 \frac{\mathrm{d}^2 Q_{11}}{\mathrm{d}y^2} = 4 Q_{11}(Q_{11}^2 + Q_{12}^2 - 1) - c\left(M_1^2 - M_2^2 \right),\label{eq:Q11} \\
    & l_1 \frac{\mathrm{d}^2 Q_{12}}{\mathrm{d}y^2} = 4 Q_{12}(Q_{11}^2 + Q_{12}^2 - 1) - 2cM_1 M_2,\label{eq:Q12}\\
    & \xi l_2 \frac{\mathrm{d}^2 M_{1}}{\mathrm{d}y^2} = \xi M_1 \left(M_1^2 + M_2^2 - 1 \right) - 2c Q_{11}M_1 - 2c Q_{12}M_2,\label{eq:M1} \\
    & \xi l_2 \frac{\mathrm{d}^2 M_{2}}{\mathrm{d}y^2} = \xi M_2 \left(M_1^2 + M_2^2 - 1 \right) + 2c Q_{11} M_2 - 2c Q_{12} M_1\label{eq:M2},
\end{align}
\end{subequations}
The first result concerns a brief proof of the existence of a global minimiser of the free energy (\ref{eq:4}), in $\mathcal{A}$.

\begin{theorem}\label{thm:existance}
    For all positive values of $(l_1, l_2, c, \xi)$, there exists at least one minimiser $\left(Q_{11}^*, Q_{12}^*, M_1^*, M_2^* \right)$ of the ferronematic free energy \cref{eq:4} in the admissible space
    \cref{eq:admissible_space}. 
    Moreover, this minimiser is a (classical) solution of the Euler--Lagrange equations \cref{eq:Q11}-\cref{eq:M2} subject to the boundary conditions \cref{eq:5}.
\end{theorem}
\begin{remark}
    For brevity of notations, we omit $(\Omega;S_0)$ and $(\Omega;\mathbb{R}^2)$ in the Sobolev spaces hereafter, whenever it causes no confusions.
\end{remark}
\begin{proof}
    The admissible space \cref{eq:admissible_space} is nonempty as $(Q_{11},Q_{12},M_1,M_2)= (-y,0,-y,0) \in \mathcal{A}$.
    The ferronematic energy \cref{eq:4} is quadratic and thus, convex in the gradient of all four state variables $(Q_{11}, Q_{12}, M_1, M_2)$ and hence, lower semicontinuous \cite{evans-2010-book}.
    Furthermore, the coupling energy can be decomposed as follows
\begin{align*}
    - cQ_{11}\left(M_1^2 - M_2^2 \right) - 2c Q_{12}M_1 M_2
    &\geq-c(M_1^2+M_2^2)(|Q_{11}|+|Q_{12}|) \\
    &\geq-\frac{c}{2}\left(\epsilon\left(|Q_{11}|+|Q_{12}|\right)^2+\frac{1}{\epsilon}\left(M_1^2+M_2^2\right)^2\right)\\
    &\geq-\frac{c}{2}\left(2\epsilon\left(Q_{11}^2+Q_{12}^2\right)+\frac{1}{\epsilon}\left(M_1^2+M_2^2\right)^2\right),
\end{align*} where $\epsilon >0$ is arbitrary.
Hence, the energy density is bounded from below as
\begin{align*}
    &\frac{l_1}{2} \left[\left( \frac{\mathrm{d} Q_{11}}{\mathrm{d} y}\right)^2 + \left(\frac{\mathrm{d} Q_{12}}{\mathrm{d} y}\right)^2 \right] + \left( Q_{11}^2 + Q_{12}^2 - 1 \right)^2  + \frac{\xi l_2}{2} \left[ \left(\frac{\mathrm{d} M_1}{\mathrm{d} y}\right)^2 + \left(\frac{\mathrm{d} M_2}{\mathrm{d} y}\right)^2 \right] \\
    &\quad + \frac{\xi}{4}\left(M_1^2 + M_2^2 - 1 \right)^2  - cQ_{11}\left(M_1^2 - M_2^2 \right) - 2c Q_{12}M_1 M_2 \\
    &\geq\frac{l_1}{2}\left[\left( \frac{\mathrm{d} Q_{11}}{\mathrm{d} y}\right)^2 + \left(\frac{\mathrm{d} Q_{12}}{\mathrm{d} y}\right)^2 \right] + \frac{\xi l_2}{2} \left[ \left(\frac{\mathrm{d} M_1}{\mathrm{d} y}\right)^2 + \left(\frac{\mathrm{d} M_2}{\mathrm{d} y}\right)^2\right]\\
    &\quad + \left[Q_{11}^2+Q_{12}^2 -\left(1+\frac{c\epsilon}{2}\right) \right]^2+ \left(\frac{\xi\epsilon-2c}{4\epsilon} \right) \left(M_1^2+M_2^2 -\frac{\epsilon\xi}{\xi\epsilon-2c}\right)^2\\
    &\quad - \left( c\epsilon + \frac{c^2\epsilon^2}{4} + \frac{c\xi}{2(\xi\epsilon-2c)}\right),
\end{align*}
and thus the full energy \cref{eq:4} is coercive provided $\epsilon>\frac{2c}{\xi}$.
The existence of a minimiser in the admissible space $\mathcal{A}$ therefore follows by the direct method in the calculus of variations \cite{evans-2010-book}. We can follow the arguments from elliptic regularity in \cite{bethuel-1993-article} and \cite{zarnescu-2010-article} to deduce that minimisers, and in fact all critical points of the free energy, are classical solutions of \cref{eq:Q11}-\cref{eq:M2}.
\end{proof}

\subsection{Maximum principle and uniqueness results}
    For simplicity and brevity, we take $l_1=l_2=l$ and $\xi=1$ hereafter.
    The cases of $l_1 \neq l_2$ and $\xi\neq 1$ can be tackled using similar mathematical methods, although $\xi$ is necessarily small for dilute ferronematic suspensions.

\begin{theorem}
    \label{thm:maximum-principle}
    (Maximum principle)
    There exists an $L^\infty$ bound for the solutions, $(Q_{11}, Q_{12}, M_1, M_2 )$ of the system \cref{eq:Q11}-\cref{eq:M2}, subject to the boundary conditions \cref{eq:5}.
    Specifically,
    \begin{equation}
        Q_{11}^2(y)+Q_{12}^2(y)\leq (\rho^*)^2,\;M_1^2(y)+M_2^2(y)\leq 1+2c\rho^*\quad\forall y\in[-1,1],
        \label{eq:bound} 
    \end{equation}
    where $\rho^*$ is given by
    \begin{equation}
    \rho^*=\left(\frac{c}{8}+\sqrt{\frac{c^2}{64}-\frac{1}{27}\left(1+\frac{c^2}{2}\right)^3}\right)^\frac{1}{3}+\left(\frac{c}{8}-\sqrt{\frac{c^2}{64}-\frac{1}{27}\left(1+\frac{c^2}{2}\right)^3}\right)^\frac{1}{3} \label{eq:rho_max}.
\end{equation}
\end{theorem}
\begin{proof}
Assume that, $|\Qvec|=\sqrt{Q_{11}^2+Q_{12}^2}$, and, $|\Mvec|=\sqrt{M_1^2+M_2^2}$, attain their maxima at two distinct points $y_1, y_2\in \left(-1,1 \right)$, respectively, then we have
$$
\frac{\mathrm{d}^2}{\mathrm{d}y^2} \left(\frac{1}{2}|\Qvec|^2\right)(y_1)\leq 0
\text{ and }
\frac{\mathrm{d}^2}{\mathrm{d}y^2} \left(\frac{1}{2}|\Mvec|^2\right)(y_2)\leq 0.
$$
Multiplying \cref{eq:Q11} by $Q_{11}$, \cref{eq:Q12} by $Q_{12}$, adding the resulting equations,
and using the identity $\frac{\mathrm{d}^2}{\mathrm{d}y^2}\left(\frac{1}{2}|\Qvec|^2\right)= \frac{\mathrm{d}^2 Q_{11}}{\mathrm{d}y^2} Q_{11}+ \frac{\mathrm{d}^2 Q_{12}}{\mathrm{d}y^2} Q_{12}+\left(\frac{\mathrm{d}Q_{11}}{\mathrm{d}y}\right)^2+\left(\frac{\mathrm{d} Q_{12}}{\mathrm{d}y}\right)^2$, we obtain the necessary condition 
\begin{equation}
    \left[
    4\left(Q_{11}^2+Q_{12}^2\right)(Q_{11}^2+Q_{12}^2-1)-c\left(Q_{11}(M_1^2-M_2^2)+2Q_{12}M_1M_2\right)
\right]
\biggr\rvert_{y=y_1} \leq 0. \label{eq:Q_condition}
\end{equation}
Similarly, we have
\begin{equation}
    \left[\left(M_1^2+M_2^2\right)\left(M_1^2+M_2^2-1\right)-2c\left(Q_{11}(M_1^2-M_2^2)+2Q_{12}M_1M_2\right)
\right]
\biggr\rvert_{y=y_2}\leq 0 \label{eq:M_condition}.
\end{equation}
Substituting
\begin{equation}
    \label{eq:substitution}
    \begin{aligned}
        &Q_{11}=\rho\cos(\theta), Q_{12}=\rho\sin(\theta),\\
        &M_1=\sigma\cos(\phi), M_2=\sigma\sin(\phi),
    \end{aligned}
\end{equation}
with $\rho=|\Qvec|\geq 0$ and $\sigma=|\Mvec|\geq 0$, with arbitrary $\theta$ and $\phi$, into \cref{eq:Q_condition} and \cref{eq:M_condition}, we obtain
\begin{align}
0 \geq \left[4\rho^2(\rho^2-1)-c\rho\sigma^2\cos(\theta-2\phi)\right]\bigg\rvert_{y=y_1} \geq
\left[4\rho^2(\rho^2-1)-c\rho\sigma^2\right]\bigg\rvert_{y=y_1} &, \nonumber \\
    \implies  \left(\rho^3-\rho-\frac{c\sigma^2}{4}\right)\bigg\rvert_{y=y_1} \leq 0 &,\label{eq:cubicrholhs}
\end{align}
and
\begin{align*}
0 \geq \left[\sigma^2(\sigma^2-1)-2c\rho\sigma^2\cos(\theta-2\phi)\right]\bigg\rvert_{y=y_2} \geq
\left[\sigma^2(\sigma^2-1)-2c\rho\sigma^2\right]\bigg\rvert_{y=y_2} &,\\
\implies \left(\sigma^2- 1 -2c\rho\right)\bigg\rvert_{y=y_2} \leq 0 &,
\end{align*}
respectively.
We then immediately deduce that $\sigma^2(y)\leq 1+2c\rho(y_2)$ for all $y\in[-1,1]$, as $|\Mvec|$ attains its maximum at $y_2$, and since $\rho(y_1)\geq\rho(y_2)$ (as $|\Qvec|$ attains its maximum at $y_1$), we further have $\sigma^2(y_1)\leq 1+2c\rho(y_1)$. 
 Using this in \cref{eq:cubicrholhs}, we get
\begin{equation*}
    0 \geq \left(\rho^3-\rho-\frac{c\sigma^2}{4}\right)\bigg\rvert_{y=y_1}\ge
    \left(\rho^3-\rho\left(1+\frac{c^2}{2}\right)-\frac{c}{4}\right)\bigg\rvert_{y=y_1},
\end{equation*}
which holds provided that $\rho$ is less than or equal to the largest positive root of the cubic polynomial, $\rho^3-\rho\left(1+\frac{c^2}{2}\right)-\frac{c}{4}$.
From the detailed calculations in the supplementary materials, the largest positive root is given by
\begin{equation*}
    \rho=
    \left(\frac{c}{8}+\sqrt{\frac{c^2}{64}-\frac{1}{27}\left(1+\frac{c^2}{2}\right)^3}\right)^\frac{1}{3}+\left(\frac{c}{8}-\sqrt{\frac{c^2}{64}-\frac{1}{27}\left(1+\frac{c^2}{2}\right)^3}\right)^\frac{1}{3},
\end{equation*}
and thus $\rho(y_1) \leq \rho^*$.
The $L^\infty$ bounds for $\rho$ and $\sigma$ are an immediate consequence, i.e.,
\begin{equation*}
	\rho(y)\leq\rho^*,\;\sigma^2(y)\leq 1+2c\rho^*\quad\forall y\in[-1,1].
\end{equation*}
Note that if $y_1=y_2$, the proof is unchanged since $\rho(y_1)=\rho(y_2)$.
\end{proof}
\begin{remark}
    For $c=0$, the upper bounds \cref{eq:bound} reduce to
    $ Q_{11}^2+Q_{12}^2\leq 1,\; M_1^2 +M_2^2\leq 1,$
   which are the Ginzburg--Landau bounds in \cite{majumdar-2010-article} for $\Qvec$ and $\Mvec$.
    Moreover, if $c$ is small, we can expand $\rho^*$ in powers of $c$ to deduce that  $Q_{11}^2+Q_{12}^2\leq 1+\frac{c}{4},\; M_1^2+M_2^2\leq 1+2c$
       to leading order in $c$.
    Hence, the nemato-magnetic coupling perturbs the Ginzburg--Landau bounds linearly, for small $c$ (see the supplementary material for the case of large $c$ too).
\end{remark}

With the $L^\infty$ bounds at hand, one can prove that there is a unique critical point of \cref{eq:4}, which is necessarily the global energy minimiser, in the $l \to \infty$ limit.

\begin{theorem}
    \label{thm:uniqness}
    (Uniqueness of minimisers for sufficiently large $l$)
    For a fixed $c$ and for
    $l_1=l_2=:l$ 
sufficiently large and $\xi=1$, there exists a unique critical point (and hence global minimiser) of the full energy \cref{eq:4}, in the admissible space \cref{eq:admissible_space}.
\end{theorem}
\begin{proof}
    We first show that the free energy \cref{eq:4} is strictly convex using the maximum principle.
    In fact, we let $\left(\Qvec,\Mvec\right),\left(\overline{\Qvec},\overline{\Mvec}\right)\in \mathcal{A}$ so that $\left(\Qvec-\overline{\Qvec}\right)\in W^{1,2}_0$ and $\left(\Mvec-\overline{\Mvec}\right)\in W^{1,2}_0$, where $W^{1,2}_0$ is the closure of $C^\infty_0$ with respect to the $W^{1,2}$-norm.
   Note that
\begin{align}
    F\left(\frac{\Qvec+\overline{\Qvec}}{2},\frac{\Mvec+\overline{\Mvec}}{2}\right)
        =&\frac{1}{2}\left[F\left(\Qvec, \Mvec\right)+F\left(\overline{\Qvec},\overline{\Mvec}\right)\right]+\int_\Omega \Biggl\{f\left(\frac{\Qvec+\overline{\Qvec}}{2},\frac{\Mvec+\overline{\Mvec}}{2}\right) \nonumber \\
        &-\frac{1}{2}\left[f\left(\Qvec,\Mvec\right) +f\left(\overline{\Qvec},\overline{\Mvec}\right) \right]-\frac{ l}{8}\left[\left( \frac{\mathrm{d} \Qvec}{\mathrm{d} y}\right) - \left(\frac{\mathrm{d} \overline{\Qvec}}{\mathrm{d} y}\right) \right]^2 \nonumber \\
        & -\frac{l}{8}\left[\left( \frac{\mathrm{d} \Mvec}{\mathrm{d} y}\right) - \left(\frac{\mathrm{d} \overline{\Mvec}}{\mathrm{d} y}\right) \right]^2 \Biggr\}~ \mathrm{d}y \nonumber\\
        \leq & \label{eq:poincare_ineq}\frac{1}{2}\left[F\left(\Qvec, \Mvec\right)+F\left(\overline{\Qvec},\overline{\Mvec}\right)\right]+\int_\Omega \Biggl\{f\left(\frac{\Qvec+\overline{\Qvec}}{2},\frac{\Mvec+\overline{\Mvec}}{2}\right) \\
        &-\frac{1}{2}\left[f\left(\Qvec,\Mvec\right) +f\left(\overline{\Qvec},\overline{\Mvec}\right) \right] \Biggr\}~ \mathrm{d}y -\frac{l}{16}\|\Qvec-\overline{\Qvec}\|^2_{L^2} \nonumber\\
        &-\frac{l}{16}\|\Mvec-\overline{\Mvec}\|^2_{L^2} \nonumber,
\end{align}
where $f$ is the bulk energy density \cref{eq:f}, and we have used the Poincar\'e inequality with the Poincar\'e constant $c_p=\frac{1}{2}$ in the last inequality.
We estimate the second partial derivatives of $f$, using the $L^\infty$ bounds above, yielding
\begin{eqnarray*}
    && \frac{\partial^2 f}{\partial Q_{1i}\partial Q_{1j}}=4\delta_{ij}\left(Q^2_{11}+Q^2_{12}-1\right)+8Q_{1i}Q_{1j}\leq 4\left(3(\rho^*)^2-1\right)\eqqcolon a_1,\\
    && \left| \frac{\partial^2 f}{\partial M_i \partial M_j}\right|\leq  B (5c\rho^*+1) \eqqcolon a_2,\\
    &&\left|\frac{\partial^2 f}{\partial Q_{1i}\partial M_j} \right|\leq A c\sqrt{1+2c\rho*}\eqqcolon a_3,
\end{eqnarray*}
for $i,j\in\{1,2\}$, where $\delta_{ij}$ is the Kronecker delta symbol and $A$, $B$ are constants independent of $c$.
Using methods parallel to \cite[Lemma 8.2]{lamy-2014-article}, we have
\begin{equation}
\label{eq:f_bound}
\begin{aligned}
    \int_\Omega \bigg\{f
 &\left(\frac{\Qvec+\overline{\Qvec}}{2},\frac{\Mvec+\overline{\Mvec}}{2}\right)-\frac{1}{2}\left[f\left(\Qvec,\Mvec\right) +f\left(\overline{\Qvec},\overline{\Mvec}\right) \right] \bigg\}~\mathrm{d}y \\
 &\leq a_1\|\Qvec-\overline{\Qvec}\|^2_{L^2}
    +a_2\|\Mvec-\overline{\Mvec}\|^2_{L^2}+a_3\|\Qvec-\overline{\Qvec}\|_{L^2}\|\Mvec-\overline{\Mvec}\|_{L^2}.
\end{aligned}
\end{equation}
Note that
\begin{align*}
         & \left\|\Qvec-\overline{\Qvec}\right\|_{L^2} \left\|\Mvec-\overline{\Mvec}\right\|_{L^2}\leq \frac{1}{2}\left(\epsilon\left\|\Qvec-\overline{\Qvec}\right\|_{L^2}^2+\epsilon^{-1}\left\|\Mvec-\overline{\Mvec}\right\|_{L^2}^2\right) \quad \forall \epsilon>0.
\end{align*}
We take $\epsilon=2$ for convenience and then substitute \cref{eq:f_bound} into \cref{eq:poincare_ineq}, so that
\begin{align*}
    F\left(\frac{\Qvec+\overline{\Qvec}}{2},\frac{\Mvec+\overline{\Mvec}}{2}\right)
    &\leq \frac{1}{2}\left[F\left(\Qvec, \Mvec\right)+F\left(\overline{\Qvec},\overline{\Mvec}\right)\right]+\left(a_1+a_3-\frac{l}{16}\right)\|\Qvec-\overline{\Qvec}\|^2_{L^2}\\
    &+\left(a_2+\frac{a_3}{4}-\frac{ l}{16}\right)\|\Mvec-\overline{\Mvec}\|^2_{L^2}.
\end{align*}
Hence, for $l>l^*(c)=\max\left\{16(a_1+a_3),4(4a_2+a_3)\right\}$, it holds that
\begin{equation*}
    F\left(\frac{\Qvec+\overline{\Qvec}}{2},\frac{\Mvec+\overline{\Mvec}}{2}\right)< \frac{1}{2}F\left(\Qvec,\Mvec\right)+\frac{1}{2}F\left(\overline{\Qvec},\overline{\Mvec}\right)
\end{equation*}
for all $\Qvec,\overline{\Qvec}\in W^{1,2}$ and $\Mvec,\overline{\Mvec}\in W^{1,2}$ such that $\Qvec\neq\overline{\Qvec}$, $\Mvec\neq\overline{\Mvec}$.
Therefore, $F$ is strictly convex.

Now assume that for $l\in(l^*,\infty)$, there exist two solutions $\left(\Qvec,\Mvec\right)$ and $\left(\overline{\Qvec},\overline{\Mvec}\right)$ of \cref{eq:Q11}-\cref{eq:M2} in the admissible space $\mathcal{A}$.
Then the mapping
$$[0,1]\ni t\mapsto F\left(t\Qvec+(1-t)\overline{\Qvec}, t\Mvec+(1-t)\overline{\Mvec}\right)$$
is $C^1$ (continuously differentiable) and its derivative vanishes at $t=0,1$.
However, this contradicts the strict convexity of $F$ and hence, the uniqueness result follows.
\end{proof}

\begin{remark}
    \label{rem:length}
The existence, uniqueness and maximum principle results work in two and three dimensions, and can be adapted to $l_1\neq l_2$ and $\xi \neq 1$.
    Recall the definitions of the dimensionless parameters in \cite{bisht-2019-article}: 
    \begin{equation}
        \label{eq:dimenless-def}
        l_1=\frac{L}{D^2|A|},\;l_2=\frac{\kappa}{D^2|\alpha|},\; c=\frac{\gamma\mu_0}{|A|}\sqrt{\frac{C}{2|A|}}\frac{|\alpha|}{\beta},
    \end{equation}
    where $A$ is the re-scaled temperature, $L$ is the nematic elastic constant, $\kappa$ is the magnetic elastic constant, $C$, $\alpha$, $\beta$ are material-dependent constants, $\gamma$ is a coupling parameter and $\mu_0$ is a universal constant.
    From \cref{thm:uniqness}, the conditions $l=l_1=l_2>l^*(c)=\max\{16(a_1+a_3),4(4a_2+a_3)\}$  guarantee the uniqueness of a solution for the system \cref{eq:Q11}-\cref{eq:M2}.
    The parameters $a_1,a_2, a_3$ grow as $c^2$ for large $c$, and thus the condition $l > l^*(c)$ is equivalent to  $\frac{c^2_1 L}{D^2 |A| \left(c_0^2 + c^2 \right)} \gg 1$ for some constants $c_0,c_1$ or $D \ll \sqrt{\frac{c^2_1 L}{|A| \left(c_0^2 + c^2\right)}}\eqqcolon c_1\frac{\xi_n}{\sqrt{c_0^2+c^2}}$, i.e., when the physical length $D$ is much smaller than an enhanced material-dependent length scale $c_1\frac{\xi_n}{\sqrt{c_0^2 + c^2}}$.
Here, $\xi_n=\sqrt{\frac{L}{|A|}}$ is the temperature-dependent nematic correlation length. 
For $c=0$, we recover the uniqueness results reported in \cite{lamy-2014-article} and \cite{canevari_majumdar_spicer2016}.
\end{remark}

\subsection{Convergence analysis for $l\to \infty$ and $l\to 0$}\label{sec:convergence_full_problem}
For a fixed $c>0$, the $l\to \infty$ limit corresponds to very narrow channels with $D \ll  \sqrt{\frac{L}{|A| (c_0^2+c^2)}}$ as discussed in \cref{rem:length}. 
 From the maximum principle \cref{thm:maximum-principle}, $\|\Qvec\|_{L^\infty}$ and $\|\Mvec\|_{L^\infty}$ are bounded independently of $l$, as shown in \cref{eq:bound}.
Furthermore, in the $l \to\infty $ limit, one can easily see that \cref{eq:Q11}-\eqref{eq:M2} reduce to the Laplace equations
\begin{equation}
    \label{eq:laplace_system}
    \begin{aligned}
        & \frac{\mathrm{d}^2 Q_{11}}{\mathrm{d}y^2}=0,&& \frac{\mathrm{d}^2 Q_{12}}{\mathrm{d}y^2}=0,\\
        & \frac{\mathrm{d}^2 M_1}{\mathrm{d}y^2}=0, && \frac{\mathrm{d}^2 M_2}{\mathrm{d}y^2}=0,
    \end{aligned}
\end{equation}
subject to \cref{eq:5}, which admits the unique solution as shown below:
\begin{equation}
    \label{eq:laplace_solutions}
    \begin{aligned}
         (\Qvec^\infty,\Mvec^\infty)=(Q^\infty_{11}, Q^\infty_{12},M^\infty_1,M^\infty_2)=(-y,0,-y,0).
    \end{aligned}
\end{equation}
In fact, \cref{eq:laplace_solutions} is an \textbf{order reconstruction} solution (OR), as introduced in \cref{sec:OR}, with linear profiles for $Q_{11}$ and $M_1$.
In the next theorem, we use the method of sub- and super-solutions as in \cite{fang-2020-article} to study the convergence of solutions of \cref{eq:Q11}-\cref{eq:M2} to $(\Qvec^\infty, \Mvec^\infty)$, as $l\to \infty$.

\begin{theorem}
    \label{thm:laplace}
    (Convergence result for $l\to \infty$)
    Assume $l > l^*$.
    Let $(\Qvec^l,\Mvec^l)$ be the unique solution of the Euler--Lagrange equations \cref{eq:Q11}-\cref{eq:M2} in  \cref{eq:admissible_space}. Then $(\Qvec^l,\Mvec^l)$ converge to $(\Qvec^\infty,\Mvec^\infty)$ as $l\to \infty$, with the following estimates:
\begin{equation}
    \forall j=1,2,\quad\|Q^l_{1j}-Q^\infty_{1j}\|_{L^\infty}\leq \alpha_1 l^{-1},\;\|M^l_{j}-M^\infty_{j}\|_{L^\infty}\leq \alpha_2 l^{-1},
\end{equation}
for positive constants $\alpha_1,\alpha_2$ independent of $l$.
\end{theorem}
\begin{proof}
      Recalling \cite[Proposition 3.1]{fang-2020-article}
      and comparing equations \cref{eq:Q11}-\cref{eq:M2} with the Laplace equations \cref{eq:laplace_system}, we have for $j=1,2$,
      \begin{subequations}
          \label{eq:inequality}
      \begin{align}
          -l^{-1}\left(4\rho^*\left(\left(\rho^*\right)^2-1\right)+c\left(1+2c\rho^*\right)\right) &\leq \frac{\mathrm{d}^2}{\mathrm{d}y^2}\left(Q^l_{1j}-Q^\infty_{1j}\right)\\
          &\leq l^{-1}\left(4\rho^*\left(\left(\rho^*\right)^2-1\right)+c\left(1+2c\rho^*\right)\right)\;\text{ in }\Omega,\nonumber\\
          Q^l_{1j}-Q^\infty_{1j}&=0\; \text{ on }\partial\Omega,\\
      -l^{-1}6c\rho^*(1+2c\rho^*)^\frac{1}{2} \leq \frac{\mathrm{d}^2}{\mathrm{d}y^2}(M^l_{j}-M^\infty_{j}) &\leq l^{-1}6c\rho^*(1+2c\rho^*)^\frac{1}{2}\; \text{ in }\Omega,\\
          M^l_{j}-M^\infty_j &=0\; \text{ on }\partial\Omega.
      \end{align}
  \end{subequations}
      Here, the $L^{\infty}$ bound \cref{eq:bound} has been used in the inequalities above.
      Let $v_k\in C^\infty(\Omega;\mathbb{R})$, $k=1,2$, be solutions of
      \begin{equation*}
          \left\{
      \begin{aligned}
          &\frac{\mathrm{d}^2 v_1}{\mathrm{d}y^2} = 4\rho^*\left(\left(\rho^*\right)^2-1\right)+c\left(1+2c\rho^*\right)\quad &&\text{in }\Omega,\\
          &\frac{\mathrm{d}^2 v_2}{\mathrm{d}y^2} = 6c\rho^*\left(1+2c\rho^*\right)^\frac{1}{2}\quad &&\text{in }\Omega,\\
          &v_k=0\quad \text{for }k=1,2 \quad&&\text{on }\partial\Omega.
      \end{aligned}
      \right.
      \end{equation*}
      Then each $v_k$ only depends on the coupling parameter $c$
      . Hence, by the classical sub- and super-solution method, $-l^{-1}v_1$ is a sub-solution and $l^{-1}v_1$ is a super-solution for each component of $(\Qvec^l-\Qvec^\infty)$, and similarly, $-l^{-1}v_2$ is a sub-solution and $l^{-1}v_2$ is a super-solution for each of the vector components of $(\Mvec^l-\Mvec^\infty)$.
      The estimates then follow and the proof is complete.
\end{proof}
In the supplementary material, we compute asymptotic expansions for $\Qvec^l$, $\Mvec^l$, for large $l$ and small $c$, complemented by numerical experiments.

Next, we consider the $l\to 0$ limit for fixed $c$, which is valid for large channel widths $D$, much greater than the nematic correlation length. 
 To this end, we rewrite the free energy \cref{eq:4} as 
\begin{equation}
    \label{eq:energy_non_negative}
    \begin{aligned}
        \frac{1}{l}F(Q_{11}, Q_{12}, M_1, M_2)
    \coloneqq & \int_{\Omega} \Bigg\{\frac{1}{2}\left[\left( \frac{\mathrm{d} Q_{11}}{\mathrm{d} y}\right)^2 + \left(\frac{\mathrm{d} Q_{12}}{\mathrm{d} y}\right)^2 \right] + \\
                                          & + \frac{1}{2} \left[ \left(\frac{\mathrm{d} M_1}{\mathrm{d} y}\right)^2 + \left(\frac{\mathrm{d} M_2}{\mathrm{d} y}\right)^2 \right] + \frac{1}{l} \Bar{f}(Q_{11},Q_{12},M_1,M_2)\Bigg\} ~\mathrm{d}y,
    \end{aligned}
\end{equation} where
\begin{equation}
    \label{eq:non-neg}
    \begin{aligned}
        \Bar{f}(Q_{11},Q_{12},M_1,M_2) &\coloneqq \left( Q_{11}^2 + Q_{12}^2 - 1 \right)^2 + \frac{1}{4}\left(M_1^2 + M_2^2 - 1 \right)^2\\
               &\quad - cQ_{11}\left(M_1^2 - M_2^2 \right) - 2c Q_{12}M_1 M_2-\alpha(c) \geq 0
    \end{aligned}
\end{equation}
and the $c$-dependent constant, $\alpha(c)$, is the minimum value of the bulk energy density. The set of minimisers of $\Bar{f}$ plays a crucial role in the analysis, and belong to the set 
\[
\mathcal{A}_{\min} \coloneqq \left\{ \left(Q_{11}, Q_{12}, M_1, M_2 \right)= \left( \rho^* \cos 2\phi, \rho^* \sin 2 \phi,  \sqrt{1 + 2 c \rho^*} \cos \phi,  \sqrt{1 + 2c \rho^*} \sin \phi \right) \right\},
\]
where $\rho^*$ is given by \cref{eq:rho_max} and $\phi$ is an arbitrary angle (see the supplementary material).
The set $\mathcal{A}_{\min}$ is clearly a continuum.

Consider the following admissible test maps for sufficiently small $l$, with $Q_{12}^t(y) = M_2^t(y) \equiv 0$ for $y\in\left[-1, 1 \right]$ and 
\begin{equation*}
    Q_{11}^t(y)  = \begin{cases}
        g(y), & y\in \left[-1, -1 + \sqrt{l}\right), \\
        \rho^*, & y\in \left( - 1 + \sqrt{l}, 1  - \sqrt{l} \right), \\
        h(y), & y\in \left( 1 - \sqrt{l}, 1 \right].
    \end{cases}
\end{equation*}
Here, $g$ linearly interpolates between $\rho^*$ and $g(-1) = 1$; $h$ linearly interpolates between $\rho^*$ and $h(1) = -1$. Similarly, we use the following test map for $M_1$:
\begin{equation*}
    M_{1}^t(y)  =
    \begin{cases}
        g^*(y), & y\in \left[-1, -1 + \sqrt{l}\right), \\
        \sqrt{1 + 2c\rho^*}, & y\in \left( - 1 + \sqrt{l}, 1  - \sqrt{l} \right), \\
        h^*(y), & y\in \left( 1 - \sqrt{l}, 1 \right].
    \end{cases}
\end{equation*}
Here, $g^*$ linearly interpolates between $\sqrt{1 + 2c\rho^*}$ and $g^*(-1) = 1$; $h^*$ linearly interpolates between $\sqrt{1 + 2c\rho^*}$ and $h^*(1) = -1$.
We have $\bar{f}\left(\rho^*, 0, \sqrt{1 + 2c\rho^*}, 0 \right) = 0$ (also see supplementary material).
It is then straightforward to check that
\begin{equation*}
    \frac{1}{l} F\left(Q_{11}^t, 0, M_1^t, 0 \right) \leq \frac{C}{\sqrt{l}}
\end{equation*}
for a positive constant $C$ independent of $l$, with $l$ small enough.
Hence, for an energy minimiser $(\Qvec^l, \Mvec^l)$ of the full energy \cref{eq:4}, we necessarily have that
\[
\frac{1}{l} F\left(Q_{11}^l, Q_{12}^l, M_1^l, M_2^l \right) \leq \frac{C}{\sqrt{l}},
\]
and hence,
\[
    \int_{-1}^{1} \Bar{f}\left( Q_{11}^l , Q_{12}^l, M_1^l, M_2^l \right)~\mathrm{d}y \leq C \sqrt{l} \to 0 \quad \textrm{as $l\to 0$.}
\]
Furthermore, since $\Bar{f}\geq 0$ by its definition \cref{eq:non-neg}, we deduce that $\Bar{f}\left(Q_{11}^l , Q_{12}^l, M_1^l, M_2^l \right) \equiv 0$ almost everywhere, as $l \to 0$.
Hence, in the $l\to 0$ limit, we expect the energy minimisers, $\left(\Qvec^l, \Mvec^l \right)$ to minimise the Dirichlet energy of $\Qvec$ and $\Mvec$ in the constrained set $\mathcal{A}_{\min}$ defined above, so that the limiting minimisers are given by:
%
\begin{equation}
    \label{eq:limit_map}
    \Qvec^0(c, y)=\rho^*(\cos(2\phi_0 (y)),\sin(2 \phi_0(y))),\;\Mvec^0(c, y)=\sqrt{1+2c\rho^*}\left(\cos(\phi_0(y)),\sin(\phi_0(y))\right),
\end{equation}
where there are two choices of $\phi_0$, dictated by the boundary conditions for $\Mvec$:
\begin{subequations}
    \label{eq:relation-theta-phi}
    \begin{align}
        & \frac{\mathrm{d}^2\phi_0}{\mathrm{d}y^2}=0,\\
        & \phi_0(-1) = 0, \phi_0(1) = \pi \quad \textrm{or} \quad \phi_0(-1) = 0, \phi_0(1) = -\pi,\\ 
        & 2\phi_0-\theta_0=2 n \pi, \quad\textrm{$n\in\mathbb{Z}$.}\label{eq:constraint-theta-phi}
    \end{align}
\end{subequations}
Here, $\theta_0$ and $\phi_0$ denote the director and magnetisation vector angles, respectively. 
In \cref{sec:num-full}, we numerically demonstrate that the energy minimisers, $\left(\Qvec^l, \Mvec^l \right)$ indeed converge to one of the two limiting maps in $\mathcal{A}_{\min}$, defined above, almost everywhere except near $y=\pm 1$ (and interior points associated with jumps in $(2\phi_0 - \theta_0)$, since $2\phi_0 - \theta_0$ is constrained to be an even multiple of $2\pi$, in the $l \to 0$ limit).
There are necessarily boundary layers near $y=\pm 1$, since the limiting maps in $\mathcal{A}_{\min}$ do not satisfy the  boundary conditions at $y=\pm 1$. We indeed have multistability in this limit.

We do not prove convergence results in the $l\to 0$ limit rigorously, since this requires a delicate $\Gamma$-convergence analysis for a vector-valued problem with four degrees of freedom, with a continuum vacuum manifold, and additional complications from the boundary conditions. This warrants a separate study in its own right.

\section{Order reconstruction solutions}
\label{sec:OR}
The results in \cref{sec:fullproblem} concern the full problem \cref{eq:Q11}-\cref{eq:M2} or ferronematic solutions with four degrees of freedom, $(\Qvec, \Mvec) = (Q_{11}, Q_{12}, M_1, M_2)$.
It is evident from the Euler--Lagrange equations \cref{eq:Q11}-\cref{eq:M2}, that we always have a branch of solutions with $Q_{12}=M_2=0$.
We refer to such solutions with only two degrees of freedom, $(\Qvec, \Mvec) = (Q_{11}, 0, M_1, 0)$ as \emph{order reconstruction} (OR) solutions.
A nematic (resp.\ magnetic) domain wall is defined to be a point $y=y^* \in (-1, 1)$ such that $\Qvec(y^*)=\left(Q_{11}(y^*), Q_{12}(y^*) \right) = 0$ (resp.\ $\Mvec(y^*) = 0$). We call these points ``walls" since they correspond to two-dimensional surfaces in the $xz$-plane. Ferronematic solutions need not have domain walls in general but OR solutions in the admissible space (\ref{eq:admissible_space}) must have domain walls because of the imposed Dirichlet conditions. There must exist an interior point $y^*\in (-1, 1)$ such that $Q_{11} (y^*) = 0$, because $Q_{11}(-1) = 1$ and $Q_{11}(1) = -1$, and $Q_{12}(y) = 0$ for all $y\in \left[-1, 1 \right]$ by definition; similar remarks apply to the domain wall in $\Mvec$.
Furthermore, domain walls in $\Qvec$ and $\Mvec$ can occur at different points, as we shall see in \cref{sec:num}.
OR solutions are special since the domain walls separate polydomains i.e., distinctly ordered domains.
In fact, recall the parameterisation \cref{eq:substitution} and note that $Q_{12} = M_2 = 0$ implies $\theta =n\pi$ (for some integer $n$) everywhere; equivalent remarks apply to $\phi$.
Hence, there is necessarily a domain wall in $\Qvec$ such that $\theta = 2n\pi$ on one side of the domain wall containing  $y=-1$, and $\theta=(2m + 1)\pi$ (for some integers $n,m$) on the other side of the domain wall containing  $y=1$;
analogously, there is a domain wall in $\Mvec$ that separates two polydomains, with $\phi=2n\pi$ and $\phi=(2m + 1)\pi$ for some integers $n$ and $m$ respectively. These domain walls are associated with jumps in $\nvec$ and the normalised magnetisation vector, $\mvec = \mathbf{M}/ |\mathbf{M}|$. The domain walls are not singularities of the $\Qvec$ and $\Mvec$-solutions, although they regularise singularities/jumps in $\nvec$ and $\mvec$. Domain walls need not be associated with jumps and could just be regular zeroes of the $\Qvec$ and $\Mvec$-fields, although such domain walls would be energetically expensive.  

We interpret OR solutions as critical points of the following OR energy (which is equivalent to \cref{eq:4} with $Q_{12} = M_2 = 0$):
\begin{equation}
    \label{eq:OR_energy}
    \begin{aligned}
        E(Q_{11}, M_1)\coloneqq & \int_{-1}^{1} \Bigg\{ \frac{l}{2} \left(\frac{\mathrm{d} Q_{11}}{\mathrm{d} y} \right)^2 + \frac{l}{2} \left( \frac{\mathrm{d} M_1}{\mathrm{d} y} \right)^2 + (Q_{11}^2 - 1)^2\\
    & + \frac{1}{4}\left(M_1^2 - 1\right)^2 - cQ_{11} M_1^2 \Bigg\} \ \mathrm{d}y,
    \end{aligned}
\end{equation}
subject to the boundary conditions
\begin{equation}
    \label{eq:OR_BCs}
    \begin{aligned}
    & Q_{11}\left(-1 \right) = M_1 \left(-1 \right) = 1,
    & Q_{11}\left(1 \right) = M_1 \left(1 \right) = -1,
    \end{aligned}
\end{equation}
in the admissible space
\begin{equation}
    \label{eq:admissible_space_OR}
    \mathcal{A}^\prime = 
    \left\{Q_{11},M_1\in W^{1,2}\left(\Omega;\mathbb{R}\right), Q_{11}\;\textrm{and }M_1\;\textrm{satisfy the boundary conditions \cref{eq:OR_BCs}}\right\}.
\end{equation}
The OR bulk energy density is given by:
\begin{equation}
    f^{OR}(Q_{11},M_1)=(Q^2_{11}-1)^2+\frac{1}{4}(M^2_1-1)^2-cQ_{11}M_1^2,
    \label{eq:OR_homogeneous_energy}
\end{equation}
Hence, OR solutions are classical solutions of the following coupled ordinary differential equations,
\begin{equation}
    \label{eq:OR_equations}
    \begin{aligned}
    & l \frac{\mathrm{d}^2 Q_{11}}{\mathrm{d}y^2} = 4 Q_{11}(Q_{11}^2 - 1) - cM_1^2, \\
    & l \frac{\mathrm{d}^2 M_1}{\mathrm{d}y^2} = M_1 (M_1^2 - 1) - 2 c Q_{11} M_1.
    \end{aligned}
\end{equation}

In general, we expect multiple OR solutions for fixed values of $(l, c)$ and the optimal OR solution is a minimiser of the energy \cref{eq:OR_energy} in $\mathcal{A}^\prime$.
We give a straightforward existence theorem below, which follows immediately from the direct method in the calculus of variations \cite{evans-2010-book}, along with some qualitative properties.

\begin{theorem}
    \label{thm:existence_OR}
    (Existence, uniqueness and maximum principle)
    For all values of $(l, c)$, there exists a minimiser, $\left(Q_{11}^*, M_1^* \right)$ of the OR energy \cref{eq:OR_energy} in $\mathcal{A}^\prime$.
    This OR minimiser, $(\Qvec^{OR}, \Mvec^{OR}) = \left(Q_{11}^*, 0, M_1^*, 0 \right)$, is a solution of the full system \cref{eq:Q11}-\cref{eq:M2}, and thus a critical point of the full energy \cref{eq:4}.
    Additionally, $(\Qvec^{OR}, \Mvec^{OR})$ is the unique critical point, and hence, global minimiser of the energy \cref{eq:4}, for fixed positive $c$ and $l$ large enough, as in \cref{thm:uniqness}.
    Moreover, we have
    \begin{equation}
        |Q_{11}(y)|\leq \rho^*,\;M_1^2(y)\leq 1+2c\rho^*\quad\forall y\in[-1,1],
        \label{eq:OR_bound} 
    \end{equation}
    where $\rho^*$ is given by \cref{eq:rho_max}.
\end{theorem}
\begin{proof}
    Clearly, the admissible space $\mathcal{A}^\prime$ is non-empty as $(Q_{11}, M_1) = (-y, -y) \in \mathcal{A}^\prime$.
    We observe that \cref{eq:OR_energy} is lower semicontinuous since it is quadratic and thus, convex in both the gradients of $Q_{11}$ and $M_1$ \cite{evans-2010-book}.
    As before, the coupling energy density can be decomposed as follows, for arbitrary $\epsilon>0$
    \begin{equation*}
        -cQ_{11}M_1^2\geq-\frac{c}{2}\left(\epsilon Q^2_{11}+\frac{1}{\epsilon}M_1^4\right).
    \end{equation*}
    Therefore, the OR energy density is bounded from below, since $f^{OR}$ is quartic in $Q_{11}$ and $M_1$ and can absorb the terms above, for a suitable choice of $\epsilon$. 
    The existence of a minimiser, $(Q_{11}^*, M_1^*)$, of the OR energy \cref{eq:OR_energy}, is immediate from \cite{evans-2010-book}.
Furthermore, this minimiser is a (classical) solution of the equations \cref{eq:OR_equations} subject to the boundary conditions \cref{eq:OR_BCs}.
It is straightforward to check that the resulting OR solution, $\left(\Qvec^{OR}, \Mvec^{OR}\right) = \left(Q_{11}^*, 0, M_1^*, 0 \right)$ is also a solution of the full system, \cref{eq:Q11}-\cref{eq:M2} in the admissible space \cref{eq:admissible_space} for all values of $(l, c)$.

Since the full energy \cref{eq:4} has a unique critical point for $l$ large enough (see \cref{rem:length}), we deduce that $\left(\Qvec^{OR}, \Mvec^{OR}\right)$ is the unique energy minimiser of \cref{eq:4},
in the $l\to \infty$ limit. The bounds \cref{eq:OR_bound} follow immediately from \cref{thm:maximum-principle}, using the $L^\infty$ bounds for $|\Qvec|$ and $|\Mvec|^2$ with $Q_{12} = M_2 = 0$.
The solution branch $\left(\Qvec^{OR}, \Mvec^{OR}\right)$ exists for all values of $\left(l, c \right)$.
This completes the proof.
\end{proof}

\subsection{Convergence analysis in the $l\to 0$ limit}

Now, we study the regime of small $l$,  which describes macroscopic domains with $D \gg \sqrt{\frac{L}{|A|c^2}}$, for fixed $c>0$. We define the set of minimisers of the OR bulk potential \cref{eq:OR_homogeneous_energy}:
\[
    \mathcal{B}^{OR} = \left\{ (Q_{11}, M_1) = \left(\rho^*, \sqrt{1 + 2c\rho^*} \right), (Q_{11}, M_1) = \left(\rho^*, -\sqrt{1 + 2c \rho^*} \right) \right\}.
\]
As for the full problem, we expect minimisers of the OR energy \cref{eq:OR_energy} to converge to the set $\mathcal{B}^{OR}$ almost everywhere, away from $y=\pm 1$.
In fact, the boundary conditions, $(Q_{11}(-1 ), M_1(-1)) = (1, 1)$ and $(Q_{11}(1), M_1(1)) = (-1, -1)$ do not belong to $\mathcal{B}^{OR}$, thus, OR energy minimisers must have boundary layers near $y=\pm 1$ in this limit. We make these heuristics more precise using $\Gamma$-convergence results, as in \cite{wang_canevari_majumdar_2019}.

Consider the rescaled OR energy
\begin{equation}
    \label{eq:OR_energy_non-negative}
    \begin{aligned}
        \frac{1}{\sqrt{l}}E(Q_{11}, M_1)\coloneqq & \int_{-1}^{1} \left\{ \frac{\sqrt{l}}{2} \left(\frac{\mathrm{d} Q_{11}}{\mathrm{d} y} \right)^2 + \frac{\sqrt{l}}{2} \left( \frac{\mathrm{d} M_1}{\mathrm{d} y} \right)^2 + \frac{1}{\sqrt{l}}\tilde{f}(Q_{11},M_1)\right\} ~\mathrm{d}y
    \end{aligned}
\end{equation}
where
\begin{equation}
    \label{eq:non-neg-OR}
        \tilde{f}(Q_{11},M_1)\coloneqq \left( Q_{11}^2 - 1 \right)^2 + \frac{1}{4}\left(M_1^2 - 1 \right)^2 - cQ_{11}M_1^2-\beta(c)
               \geq 0,
\end{equation}
and the $c$-dependent constant, $\beta(c)$, is the minimum value of the OR bulk potential. 
As in \cite{braides} and \cite{wang_canevari_majumdar_2019}, we let $\mathbf{p} = \left(Q_{11}, M_1 \right)$ and define the following metric $d$ in the $\mathbf{p}$-plane, for any two points $\pvec_0, \pvec_1 \in \mathbb{R}^2$:
\begin{equation}
    \label{eq:metric}
    d\left(\pvec_0, \pvec_1 \right) = \inf \left\{ \int_{-1}^{1} \tilde{f}^{1/2}\left(\mathbf{p}(t)\right) \left| \frac{\mathrm{d}\mathbf{p}(t)}{\mathrm{d}t} \right|~\mathrm{d}t: \pvec(t)\in C^{1}\left([-1,1]; \mathbb{R}^2 \right), \pvec(-1) = \pvec_0, \pvec(1) = \pvec_1 \right\}.
\end{equation}
This metric is degenerate as $\tilde{f}(\pvec) = 0$ for $\pvec=\pvec^* = \left(\rho^*, \sqrt{1 + 2c\rho^*} \right)$ and $\pvec=\pvec^{**} = \left(\rho^*, -\sqrt{1 + 2c\rho^*} \right)$.
Despite such degeneracy, the infimum in \cref{eq:metric} is indeed attained for arbitrary $\pvec_0$ and $\pvec_1$ (see \cite[Lemma 9]{braides} and \cite{wang_canevari_majumdar_2019}).
Denote $\pvec_b(1) = (-1,-1)$ and $\pvec_b(-1) = (1, 1)$.
Let $\pvec_l$ be a minimiser of \cref{eq:OR_energy_non-negative} for a fixed $c>0$.
A straightforward application of \cite[Proposition 4.1]{wang_canevari_majumdar_2019} yields the following theorem.

\begin{theorem}
    \label{thm:OR_limit_maps}
    There exists a subsequence $l_k \to 0$ such that the minimisers $\pvec_{l_k}$ of \cref{eq:OR_energy_non-negative} converge in $ L^1\left([-1,1]\right)$ almost everywhere to a map of the form
    $
    \pvec_0 = \sum_{j=1}^{N}  p^j \chi_{E_j}
    $
    where where for any j, either $p^j = p^*$ or $p^j = p^{**}$, $\chi$ is the characteristic function of an interval, $E_j \subset (-1, 1)$ such that $\cup_{j=1}^N E_j = \left(-1, 1\right)$.
    Moreover, the intervals $E_j$ minimise the following functional
    \begin{equation}
        \label{eq:gamma1}
        J[ E_j]: = \sum_{j=1}^{N-1} d(\pvec^*, \pvec^{**}) + d\left(\pvec_0, \pvec_b(-1) \right) + d\left(\pvec_0 , \pvec_b (1) \right),
    \end{equation}
where the first term describes the number of jumps between $\pvec^*$ and $\pvec^{**}$, referred to as interior transition layers that necessarily contain a magnetic domain wall, and the energetic costs of the boundary layers are captured by the second and third terms.
\end{theorem}

We compute the following transition costs
\begin{equation}
    \label{eq:metrics}
    d(\pvec^*,\pvec^{**}), d(\pvec^*,\pvec_b(1)), d(\pvec^{**},\pvec_b(-1)), d(\pvec^*,\pvec_b(-1)), d(\pvec^{**},\pvec_b(1)).
\end{equation}
using the metric \cref{eq:metric}, and we can see from \cref{fig:metric-profile} that
\begin{equation*}
   d(\pvec^{*},\pvec_b(-1))< d(\pvec^{**},\pvec_b(-1))< d(\pvec^{**},\pvec_b(1)) < d(\pvec^*,\pvec^{**}) < d(\pvec^*,\pvec_b(1)).
\end{equation*}
It is clear that the minimiser of $J$ in \cref{eq:gamma1} is $\pvec^*$, with boundary layers near the edges $y=\pm 1$ and no interior jumps between $\pvec^*$ and $\pvec^{**}$.
%

\begin{figure}
    \centering
    \begin{minipage}{0.32\textwidth}
        \centering
        \includegraphics[width=1.0\textwidth]{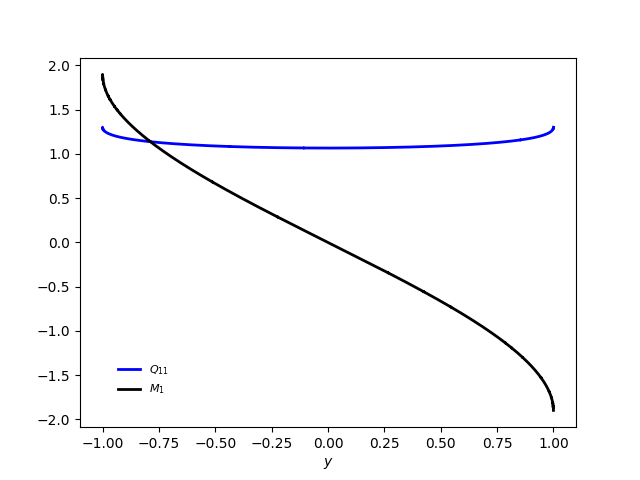}\\
$d(\pvec^*,\pvec^{**})\approx 3.008.$
    \end{minipage}
    \begin{minipage}{0.32\textwidth}
        \centering
        \includegraphics[width=1.0\textwidth]{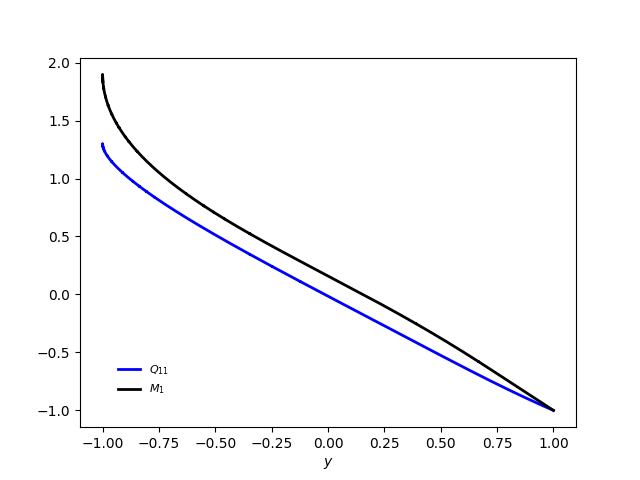}\\
        $d(\pvec^*,\pvec_b(1))\approx 3.967.$
    \end{minipage}
    \begin{minipage}{0.32\textwidth}
        \centering
      \includegraphics[width=1.0\textwidth]{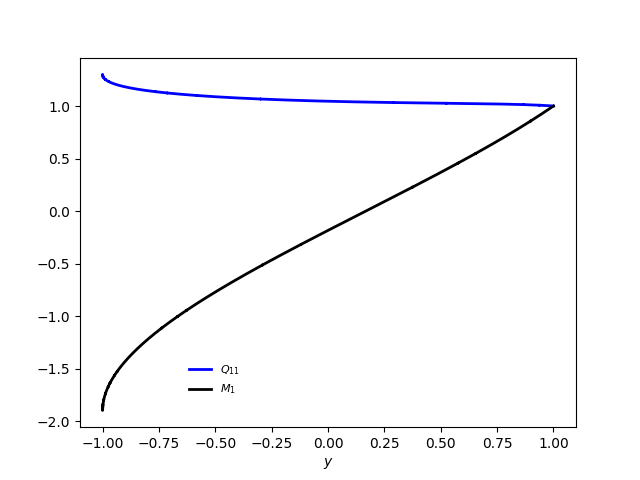}\\
        $d(\pvec^{**},\pvec_b(-1))\approx 2.577.$
    \end{minipage}
    \begin{minipage}{0.32\textwidth}
        \centering
      \includegraphics[width=1.0\textwidth]{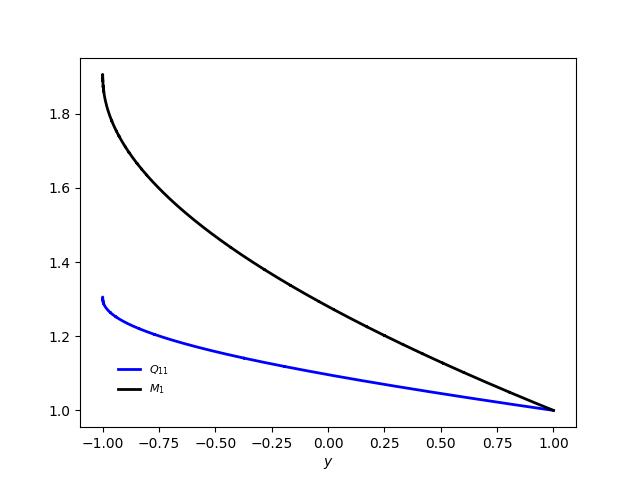}\\
        $d(\pvec^{*},\pvec_b(-1))\approx 0.455.$
    \end{minipage}
    \begin{minipage}{0.32\textwidth}
        \centering
      \includegraphics[width=1.0\textwidth]{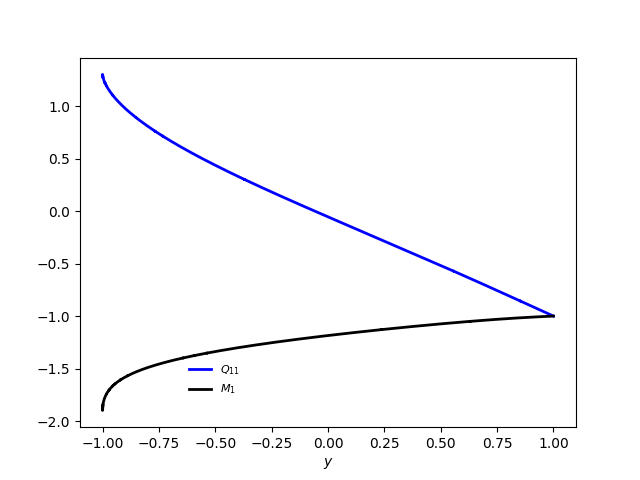}\\
        $d(\pvec^{**},\pvec_b(1))\approx 2.591.$
    \end{minipage}
    \caption{The profiles of $\pvec$ and their corresponding transition costs in \cref{eq:metrics}.}
    \label{fig:metric-profile}
\end{figure}




\subsection{Stability of OR solutions}
The authors in \cite{lamy-2014-article} and \cite{canevari_majumdar_spicer2016} consider a similar OR problem with $c=0$, 
in a one-dimensional channel and a two-dimensional square, respectively.
In both cases, the OR solution loses stability as $l$ decreases, or equivalently as the physical channel width $D$ increases, with respect to perturbations that have non-zero $Q_{12}$. This motivates us to expect a similar instability result in the ferronematic setting with $c>0$.

\begin{theorem}
    \label{thm:instability-OR}
    (Instability of the OR solution)
    For sufficiently small $l$ and a fixed positive $c$, the OR energy minimiser, $(\Qvec^{OR},\Mvec^{OR})$, is an unstable critical point of  \cref{eq:4}, in the admissible space \cref{eq:admissible_space}.
\end{theorem}
\begin{proof}
    For the OR solution $(\Qvec^{OR}, \Mvec^{OR}) = (Q_{11}^*, 0, M_1^*, 0)$, we note that $(Q_{11}^*, M_1^*)$ is a minimiser of the OR energy \cref{eq:OR_energy}.
    Furthermore, the OR solution depends on $l$ with fixed $c>0$ and we suppress this explicit dependence for brevity.
    We compute the second variation of the free energy \cref{eq:4} about $\left(\Qvec^{OR}, \Mvec^{OR} \right)$ with arbitrary perturbations,
    \begin{equation*}
        \begin{aligned}
        &\tilde{Q}_{11}(y) = Q_{11}^*(y)+ t g(y),\;
        \tilde{Q}_{12}(y) = t h(y),\;\\
        &\tilde{M}_1(y) = M_1^*(y) + t v(y),\;
        \tilde{M}_2(y) =  t w(y).
        \end{aligned}
    \end{equation*}
    Here, $t\in\mathbb{R}$ and $g(y) = h(y) = v(y) = w(y) = 0$ at $y=\pm 1$.
    The second variation is then given by
    \begin{equation}
    \label{eq:2nd_variation}
    \begin{aligned}
    \delta^2 F[g,h,v,w] &\coloneqq \frac{\mathrm{d}^2}{\mathrm{d}t^2}\biggr|_{t=0}F(\tilde{Q}_{11},\tilde{Q}_{12},\tilde{M}_1,\tilde{M}_2)\\
                        &=\delta^2 E [g,v] + \int_{-1}^{1} \Bigg\{ l\left(\frac{\mathrm{d}h}{\mathrm{d}y} \right)^2 +  l\left(\frac{\mathrm{d}w}{\mathrm{d}y} \right)^2 + 4h^2((Q_{11}^*)^2 - 1) \\
                         & \quad + w^2 ((M_1^*)^2 - 1) + 2c w^2 Q_{11}^* - 4 c h w M_1^*\Bigg\}~\mathrm{d}y\\
                         &\eqqcolon \delta^2 E [g,v] + H[h,w],
    \end{aligned}
    \end{equation}
    where $\delta^2 E[g,v]$ is the second variation of the OR energy \cref{eq:OR_energy} about $(Q_{11}^*, M_1^*)$, and thus necessarily non-negative for all admissible $(g,v)$.
    To demonstrate the instability of $(\Qvec^{OR}, \Mvec^{OR})$, we need to construct non-trivial $h$ and $w$ such that $H[h,w]<0$.
    To this end, we follow methods in \cite{lamy-2014-article} and choose
    \begin{equation}
        \label{eq:h_w}
        h(y) = \frac{\mathrm{d} Q_{11}^*}{\mathrm{d}y} z(y)\eqqcolon \tau(y)z(y), \qquad w(y) = \frac{\mathrm{d} M_1^*}{\mathrm{d}y} z(y)\eqqcolon \zeta(y)z(y),
    \end{equation}
    where $z$ is a smooth cut-off function with bounded derivatives (independent of $l$) and $z(y)=0$ for $|y| > 1 - \eta$, $0<\eta < \frac{1}{4}$. 
Since $h$ and $w$ vanish at $y=\pm 1$, we have
\[
    \int_{-1}^{1} \left(h^{'} \right)^2~ \mathrm{d}y = -\int_{-1}^{1} h h^{''} 
    ~\mathrm{d}y \quad 
\textrm{and} \quad 
\int_{-1}^{1} \left(w^{'} \right)^2~\mathrm{d}y = -\int_{-1}^{1} w w^{''} 
    ~\mathrm{d}y.
\]
Here and hereafter, we take $'$ (resp.\ $''$) to denote first (resp.\ second) derivative with respect to $y$.
    Furthermore, one can check from \cref{eq:OR_equations} that
    \begin{equation}
    \label{eq:h_w1}
    \begin{aligned}
    & 
    \tau^{'}= \frac{1}{l} \left[ 4Q_{11}^*\left((Q^*_{11})^2 - 1 \right) - c (M_1^*)^2 \right],
    &&
    \tau^{''}= \frac{1}{l} \left[ 4\tau\left(3 (Q^*_{11})^2 - 1 \right) - 2 c M_1^* \zeta \right],\\
    &
    \zeta^{'} = \frac{1}{l} \left[ M_1^*\left((M_1^*)^2 - 1 \right) - 2 c Q_{11}^* M^*_1 \right],
    &&
    \zeta^{''} = \frac{1}{l} \left[ \zeta\left(3 (M_1^*)^2 - 1 \right) - 2 cM^*_1 \tau - 2c Q_{11}^*\zeta \right].
    \end{aligned}
\end{equation}
Now noting $h^{''} = \tau^{''} z + 2 \tau^{'} z^{'} + \tau z^{''}$, $w^{''} = \zeta^{''} z +2\zeta^{'}z^{'}+\zeta z^{''}$ and substituting \cref{eq:h_w} and \cref{eq:h_w1} into $H[h,w]$, we obtain
\begin{equation}
    \label{eq:h_w2}
\begin{aligned}
    H[h,w]= &\int_{-1}^{1}\left\{ -8 (Q_{11}^*)^2 \tau^2 z^2 + 2\zeta^2 z^2 \left(2c Q_{11}^* - (M_1^*)^2 \right) \right\}~\mathrm{d}y \\
                    & + l\int_{-1}^{1}\left\{ -2 z z^{'} \tau \tau^\prime - 2 z z^{'}\zeta \zeta^\prime\right\}~\mathrm{d}y+ \int_{-1}^{1}\left\{ - l z z^{''} \left(\tau^2 + \zeta^2 \right)\right\}~\mathrm{d}y \\
    \eqqcolon & H_1+H_2+H_3.
\end{aligned}
\end{equation}

The $\Gamma$-convergence result in \cref{thm:OR_limit_maps} implies that for an interior interval $(a,b) \subset [-1,1]$, it holds that
\begin{equation}
    \label{eq:gamma-conv-result}
\int_{a}^{b} \left|{Q}_{11}^* - \rho^* \right|~\mathrm{d}y \to 0 \quad \text{and }
\int_{a}^{b} \left|({M}_1^*)^2 - 1 - 2c\rho^* \right|~\mathrm{d}y \to 0 \quad \textrm{as $l \to 0$}.
\end{equation}
We use integration by parts to obtain (recall that $l \int_{-1}^{1} \tau^2 + \zeta^2 dy \leq C\sqrt{l}$ as $l \to 0$ from the work in Section~$2.2$): 
\begin{equation*}
    \int_{-1}^1 \left\{ z z^{'}\tau \tau^{'}+z z^{'}\zeta \zeta^{'} \right\} ~\mathrm{d}y
    =-\frac{1}{2}\int_{-1}^{1} \left\{ \left(z^{'}\right)^2\left(\tau^2+\zeta^2\right)+z z^{''}\left(\tau^2+\zeta^2\right) \right\} ~\mathrm{d}y,
\end{equation*}
so that
$
    H_2 \to 0 \quad \text{as }l\to 0.
$
Moreover, it is straightforward to see that the third integral $H_3$ in \cref{eq:h_w2} vanishes in the $l\to 0$ limit.
It remains to estimate the first integral in \cref{eq:h_w2}.
By \cref{eq:gamma-conv-result}, we deduce that
\begin{equation*}
    H_1 \to \int_{-1}^{1} \left\{-8 \tau^2 z^2\left(\rho^*\right)^2 -2\zeta^2 z^2 \right\}~\mathrm{d}y < 0 \quad \text{as }l\to 0. 
\end{equation*}
\end{proof}

\section{Numerical results}
\label{sec:num}
In this section, we perform numerical experiments to validate our theoretical results and understand the interplay between $l$ and $c$ for the solution landscapes, with fixed $\xi=1$.
    For the visualisation, we plot the director $\mathbf{n}$ as rods and the normalised magnetisation vector field $\mvec=\frac{\Mvec}{|\Mvec|}$ as arrows.

\subsection{Solver details}

Since the boundary-value problem is nonlinear, we use Newton's method with $L^2$ linesearch \cite[Algorithm 2]{brune2015} as the outer nonlinear solver.
The nonlinear solver is deemed to have converged when the Euclidean norm of the residual falls below $10^{-8}$, or reduces from its initial value by a factor
of $10^{-6}$, whichever comes first.
For the inner solver, the linearised systems are solved using the sparse LU factorisation library MUMPS \cite{mumps}.
The solver described above is implemented in the \emph{Firedrake} \cite{firedrake} library, which relies on PETSc \cite{petsc} for solving the resulting linear systems.
Furthermore, we use the \emph{deflation} technique as described in \cite{farrell-birkisson-2015-article} to compute multiple solutions and bifurcation diagrams.
Throughout this section, we partition the whole interval $[-1,1]$ into $1000$ equi-distant subintervals and numerically approximate the solutions using $\mathbb{P}^1$ finite elements (piecewise linear continuous polynomials).

\textbf{Code availability.} For reproducibility and more details of the implementation, we have archived the solver code \cite{zenodo-ferronematics} and the exact version of Firedrake \cite{zenodo/firedrake-ferronematics} used to produce the numerical results of this work. An installation of Firedrake with components matching those used in this paper can be obtained by following the instructions at \url{https://www.firedrakeproject.org/download.html} with
\begin{verbatim}
    python3 firedrake-install --doi 10.5281/zenodo.4449535
\end{verbatim}
Defcon version \#aaa4ef should then be installed, as described in \url{https://bitbucket.org/pefarrell/defcon/}.

\subsection{OR solutions}
We have analysed the OR solution branch with $Q_{12} = M_2 = 0$, as $l\to 0$ and as $l\to \infty$.
The OR branch is fully characterised by solutions of the boundary-value problem \cref{eq:OR_equations}.
OR solutions are special since they must contain separate domain walls in $\Qvec$ and $\Mvec$, which can be tailored by varying $l$ and $c$. 

As $l \to \infty$, recall \cref{thm:existence_OR} to deduce that the OR solution branch is approximately given by
$\left(\Qvec^{OR}, \Mvec^{OR}\right) \approx (-y, 0, -y, 0)$, for a fixed $c$, and that $\left(\Qvec^{OR}, \Mvec^{OR}\right)$ is also the unique minimiser of both the OR energy \cref{eq:OR_energy} and the full energy \cref{eq:4}.
In \cref{fig:l10}, we plot the OR solution of \cref{eq:OR_equations} for $c=1$ and $l=10$.
The profile is indeed linear, and we do not numerically obtain any other solutions, supporting the uniqueness result in the $l\to\infty$ limit.
The OR solution vanishes at the channel centre $y=0$, i.e.~$Q_{11}(0) = M_1 (0) = 0$, and thus both the nematic and magnetic domain walls coincide at $y=0$.
Therefore, the normalised magnetisation vector $\mvec$ and director $\nvec$ have a jump discontinuity at $y=0$, i.e., $\mvec$ jumps from $\mvec=(1,0)$ for $y<0$ to $\mvec=(-1,0)$ for $y>0$, while $\nvec$ jumps from $\nvec=(1,0)$ (modulo a sign) for $y<0$ to $\nvec = (0,1)$ (modulo a sign) for $y>0$. 
We also plot the pointwise $L^\infty$ bound \cref{eq:OR_bound} as blue solid lines in \cref{fig:l10}, and this bound is indeed respected.

\begin{figure}[!ht]
    \centering
    \begin{minipage}{1.0\textwidth}
        \centering
        \includegraphics[scale=0.35]{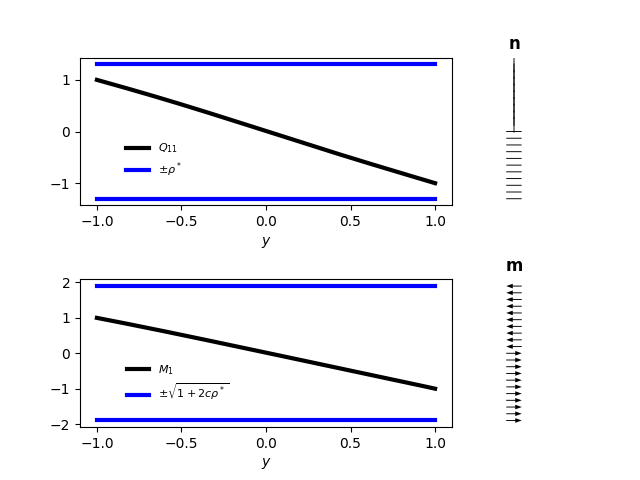}
    \end{minipage}
    \caption{The only (stable) solution of \cref{eq:OR_energy} for $c=\xi=1$, and $l=10$.}
    \label{fig:l10}
\end{figure}

As $l\to 0$, for fixed $c>0$, we expect $Q_{11} \to \rho^*$ and $M_1^2 \to 1 + 2c \rho^*$ uniformly away from $y=\pm 1$, for the OR energy minimiser in \cref{eq:OR_energy}. 
We note that $\rho^*(c) $ defined in \cref{eq:rho_max} is an increasing function of $c$ and $\rho^*(0)=1$, thus $\rho^*(c) > 1$ for all $c>0$.
As discussed in \cref{thm:OR_limit_maps}, we expect a domain wall in $\Qvec$ near the edge $y=1$, within a boundary layer of width $\sqrt{l}$, where $Q_{11}$ jumps from $Q_{11}=\rho^*>1$ to $Q_{11}(1) = -1$.
Hence, there necessarily exists a nematic domain wall with $Q_{11}=0$, within this boundary layer close to  $y=1$.
Analogously, there is a boundary layer near the other end point $y=-1$, within which $Q_{11}$ jumps from $Q_{11}(-1)=1$ to $Q_{11} =\rho^*$, but this boundary layer does not contain a nematic domain wall.
Moreover, we expect that there are at least two minimisers of the OR energy \cref{eq:OR_energy} for $l$ small enough, with opposite signs of $M_1$ in the channel interior. 
Each of these minimisers must contain at least one magnetic domain wall: near $y=1$ if $M_1 >0$ in the interior, or near $y=-1$ if $M_1<0$ in the interior respectively.
In what follows, a transition layer refers to a thin interval within which $M_1$ jumps between $-\sqrt{1 + 2c\rho^*}$ and $\sqrt{1 + 2c\rho^*}$ and each of these transition layers necessarily contains a magnetic domain wall with $M_1 = M_2 =0$.
We expect the OR energy \cref{eq:OR_energy} to have multiple critical points, with multiple interior transition layers and domain walls in $\Qvec$ and $\Mvec$, for $l$ small enough. However, we only expect two OR energy minimisers, that have the same $Q_{11}$ profile but differ in the sign of $M_1$, and the nematic and magnetic domain walls do not occur at the same point. Of course, all OR solutions are unstable critical points of the full energy \cref{eq:4} for $l$ small enough, as proven in \cref{thm:instability-OR}.
We now numerically corroborate these theoretical conjectures with $l=0.01$ and $\xi=1$.

In \cref{fig:c1}, we present four example solutions with $c=1$.
In fact, they are all unstable critical points of the full energy \cref{eq:4} whilst being stable critical points of the OR energy \cref{eq:OR_energy} (in the sense that the Hessian of second variation of the OR energy about these critical points has positive eigenvalues).
As expected, these solution profiles, $(Q_{11}, M_1)$, have boundary layers near the end points
. Furthermore, interior transition layers in $M_1$ (near the centre $y=0$) are observed in Solutions $3$ and $4$.
The $L^\infty$ bounds \cref{eq:OR_bound} (blue solid line) for $|Q_{11}|$ and $|M_1|$ are also satisfied.

\begin{figure}[!ht]
    \centering
    \begin{minipage}{0.45\textwidth}
        \centering
        \includegraphics[width=1.0\textwidth]{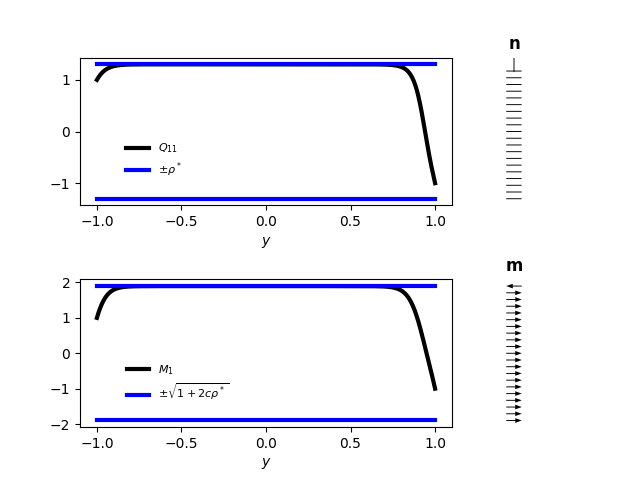}\\
        (Solution $1$)
    \end{minipage}
    \begin{minipage}{0.45\textwidth}
        \centering
        \includegraphics[width=1.0\textwidth]{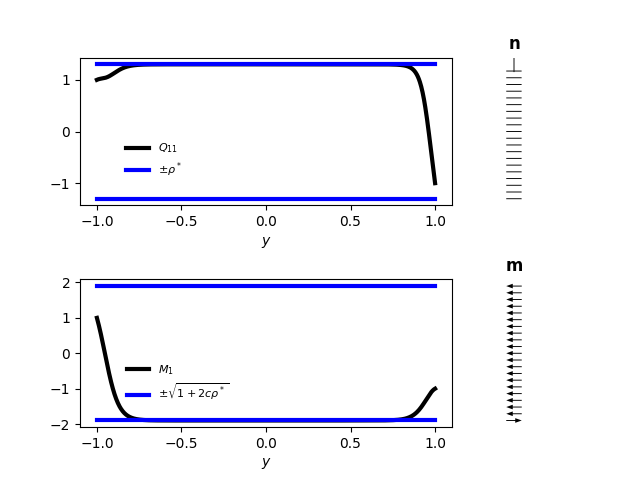}\\
        (Solution $2$)
    \end{minipage}
    \begin{minipage}{0.45\textwidth}
        \centering
        \includegraphics[width=1.0\textwidth]{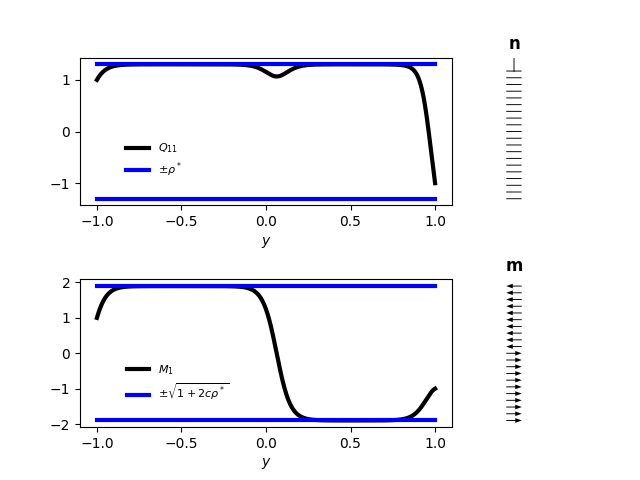}\\
        (Solution $3$)
    \end{minipage}
    \begin{minipage}{0.45\textwidth}
        \centering
        \includegraphics[width=1.0\textwidth]{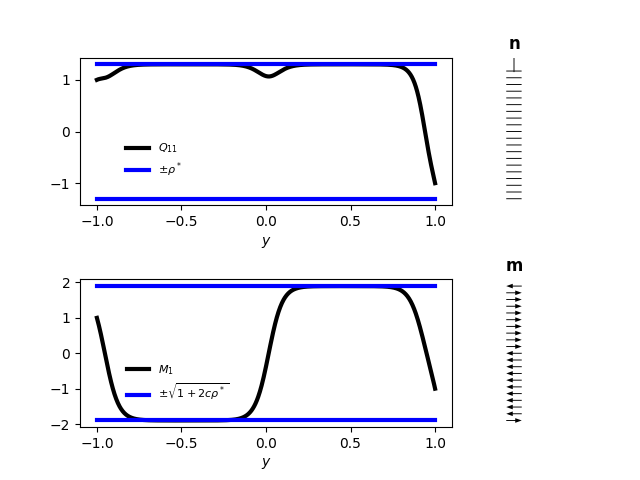}\\
        (Solution $4$)
    \end{minipage}
    \caption{Four stable OR critical points of \cref{eq:OR_energy}, 
        with $c=\xi=1$ and $l=0.01$.
         Solution $1$ is the minimiser of the OR energy \cref{eq:OR_energy}.
    }
    \label{fig:c1}
\end{figure}

In \cref{fig:c5}, we plot the stable solutions of the OR energy \cref{eq:OR_energy}, for a larger value $c=5$,
which are unstable critical points of the full energy \cref{eq:4}.
Indeed, each of the solutions in \cref{fig:c5} has one unstable eigendirection, in the context of the full energy \cref{eq:4}.
The two profiles in \cref{fig:c5}, have boundary layers near $y=\pm 1$, and essentially differ in the sign of $M_1$ in the interior; $Q_{11}$ only vanishes near $y=1$ as predicted by the $\Gamma$-convergence analysis, so that we have a nematic domain wall near $y=1$
. On the other hand, $M_1$ can vanish either near $y=-1$ or near $y=1$, so that the corresponding magnetic domain wall can occur near either boundary.
Additionally, there are other solutions with interior transition layers for $M_1$, see two examples in \cref{fig:c5-l001-multilayer} where single and multiple interior transition layers in $M_1$ are observed.
They are also stable critical points of the OR energy \cref{eq:OR_energy}. 
The transition layers in $M_{1}$ necessarily contain a magnetic domain wall with $M_1 =0$, and these interior magnetic domain walls are not accompanied by associated nematic domain walls.
Moreover, solutions with interior transition layers have higher OR energy \cref{eq:OR_energy} than solutions without interior transition layers in \cref{fig:c5}, since each transition layer has an energetic cost of $d(\pvec^*, \pvec^{**})$ as explained in \cref{thm:OR_limit_maps}.

\begin{figure}[!ht]
    \centering
    \begin{minipage}{0.45\textwidth}
        \centering
        \includegraphics[width=1.0\textwidth]{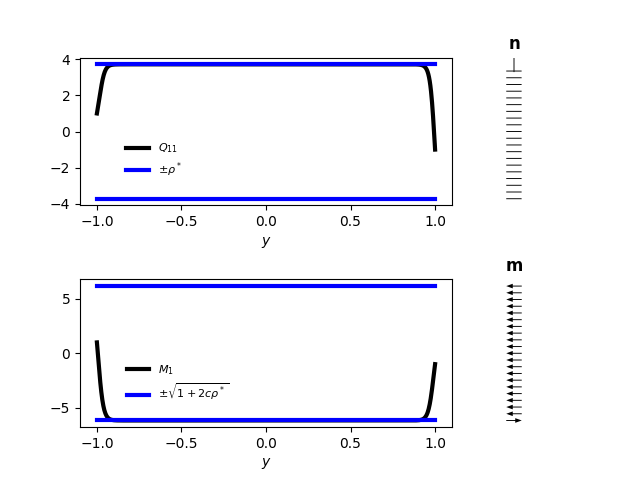}
    \end{minipage}
    \begin{minipage}{0.45\textwidth}
        \centering
        \includegraphics[width=1.0\textwidth]{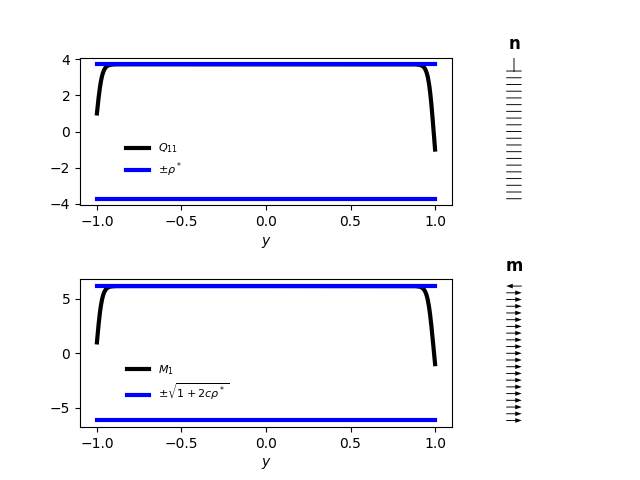}
    \end{minipage}
    \caption{Two stable OR critical points of \cref{eq:OR_energy}, for $c=5$, $\xi=1$ and $l=0.01$.
The right profile has lower OR energy than the left profile and the solutions in \cref{fig:c5-l001-multilayer}.}
    \label{fig:c5}
\end{figure}

\begin{figure}[!ht]
    \centering
    \begin{minipage}{0.44\textwidth}
        \centering
        \includegraphics[width=1.0\textwidth]{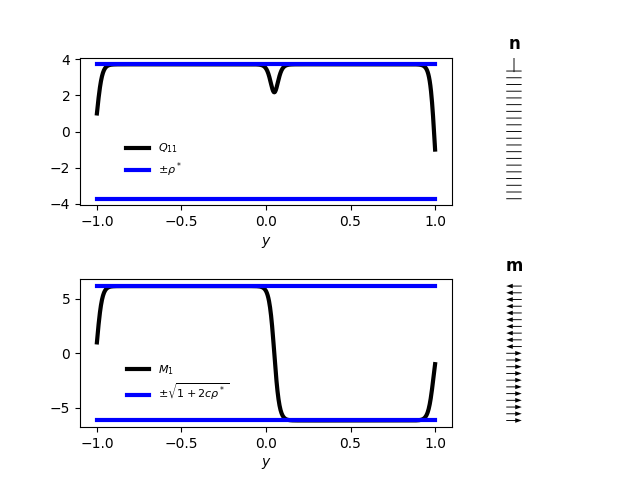}
    \end{minipage}
    \begin{minipage}{0.44\textwidth}
        \centering
        \includegraphics[width=1.0\textwidth]{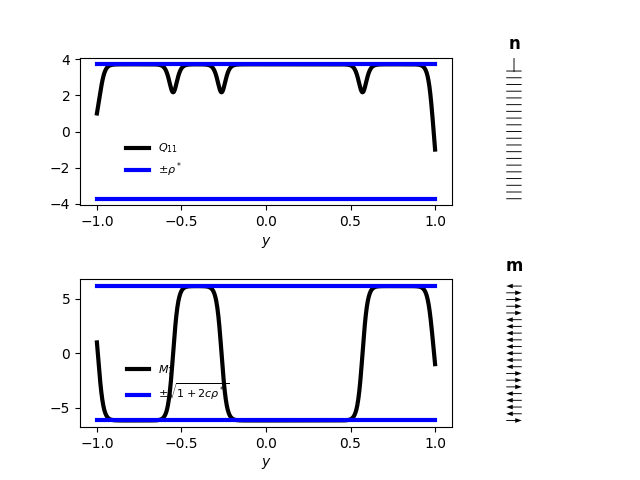}
    \end{minipage}
    \caption{Two stable OR critical points of \cref{eq:OR_energy}, with single (left) and multiple (right) interior transition layers for $c=5$, $\xi=1$ and $l=0.01$.
    The left has lower OR energy.
}
    \label{fig:c5-l001-multilayer}
\end{figure}


These numerical experiments illustrate that we can manipulate the location and multiplicity of nematic and magnetic domain walls in the OR solutions by varying $l$, e.g., the domain walls in the OR energy minimisers migrate from the channel centre to the channel boundaries at $y=\pm 1$, as $l$ decreases.


\subsection{Solutions of the full problem}
\label{sec:num-full}
Next, we consider the full problem \cref{eq:Q11}-\cref{eq:M2} with four degrees of freedom, $\left(Q_{11}, Q_{12}, M_1, M_2\right)$.
We only consider the case of small $l_1=l_2=l = 0.01$ with $\xi=1$, since the OR solution branch is the unique solution of the full problem, in the $l\to \infty$ limit. 

In \cref{fig:c1-full}, we take $c=1$ and present four examples of stable solutions with four degrees of freedom. 
We also plot the $L^\infty$ bound \cref{eq:bound} in the figures, illustrating that \cref{thm:maximum-principle} is indeed satisfied.
There are no interior domain walls with $|\Qvec|=|\Mvec|=0$, for small $l$, as discussed in \cref{sec:convergence_full_problem}.
Furthermore, Solutions $1$, $2$ and $3$ in \cref{fig:c1-full} only have boundary layers, with almost constant $|\Qvec|, |\Mvec|$-profiles in the domain interior, whereas Solution $4$ has interior non-zero local minima in $|\Qvec|$ and $|\Mvec|$.
Solutions $1$ and $2$ are the energy minimisers while the remaining two profiles are non-minimising stable critical points of the full energy \cref{eq:4}.
Note that the two energy minimisers differ in their $\mvec$-patterns (more precisely, the sign of $M_2$).
Moreover, we compute the values of $\theta$, $\phi$ defined to be
\begin{equation}
    {\theta} = \arctan\left(\frac{Q_{12}}{Q_{11}}\right), \; {\phi}=\arctan\left(\frac{M_2}{M_1}\right)
\end{equation}
for each solution. 
It can be seen that $|\Qvec| \to \rho^*$ and $|\Mvec|^2 \to 1 + 2c\rho^*$ for the energy minimiser (Solution $1$),
whereas $(2\phi - \theta)$ tends to an even multiple of $\pi$ almost everywhere, except near $y=\pm 1$.
We do not attempt to explain the interior jumps in the plots of $(2\phi - \theta)$, except that these jumps will have a distinct optical signature in physical experiments.
Furthermore, the separate plots of $\theta$ and $\phi$ demonstrate linear behaviour except around the local minima of $|\Qvec|$ or $|\Mvec|$ and the boundary layers, consistent with the limiting Laplace equation \cref{eq:relation-theta-phi} for $\phi$ and $\theta$, in the $l \to 0$ limit.

\begin{figure}[!ht]
    \centering
    \begin{minipage}{0.49\textwidth}
        \centering
        \includegraphics[width=0.80\textwidth]{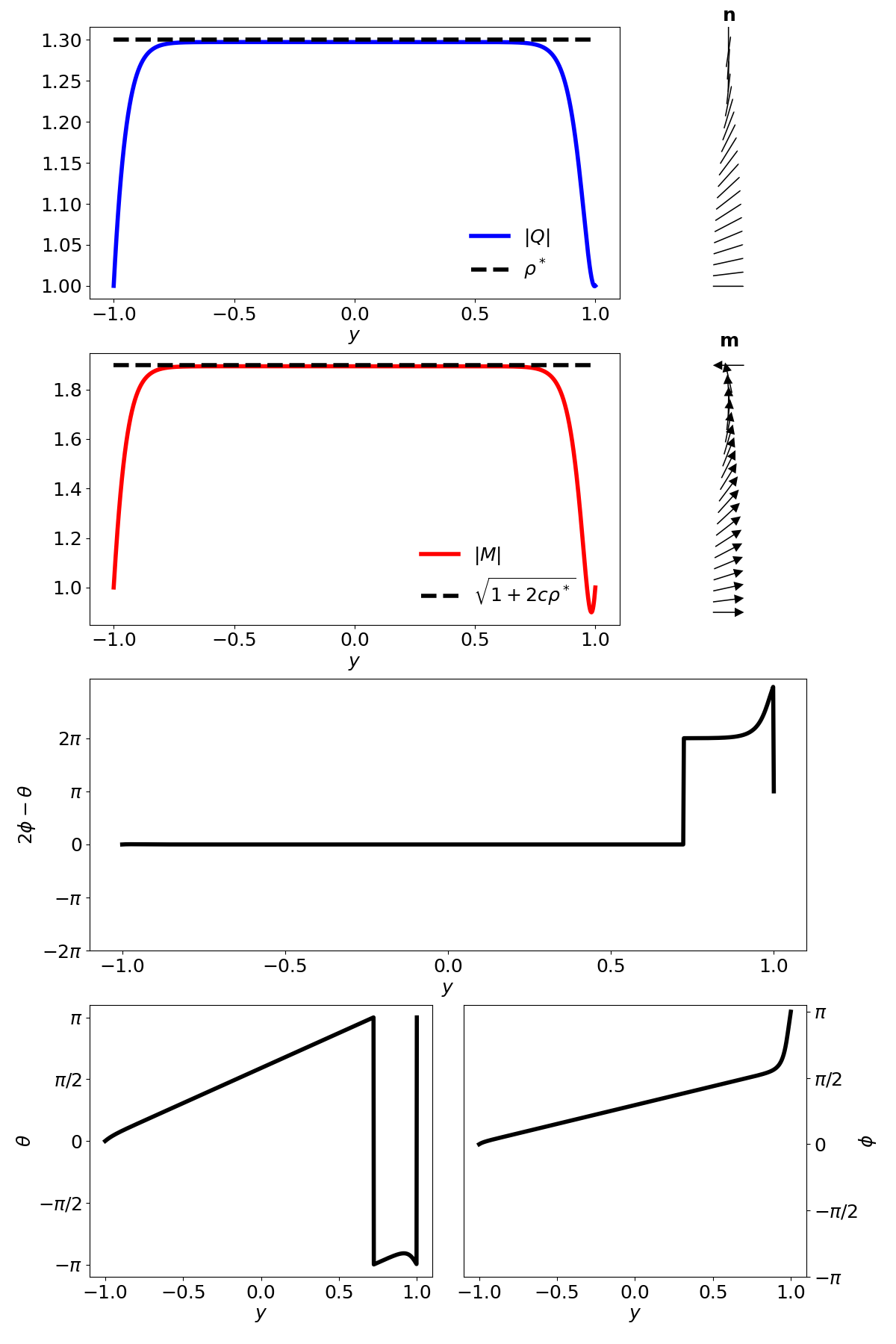}\\
        (Solution $1$; stable)
    \end{minipage}
    \begin{minipage}{0.49\textwidth}
        \centering
        \includegraphics[width=0.80\textwidth]{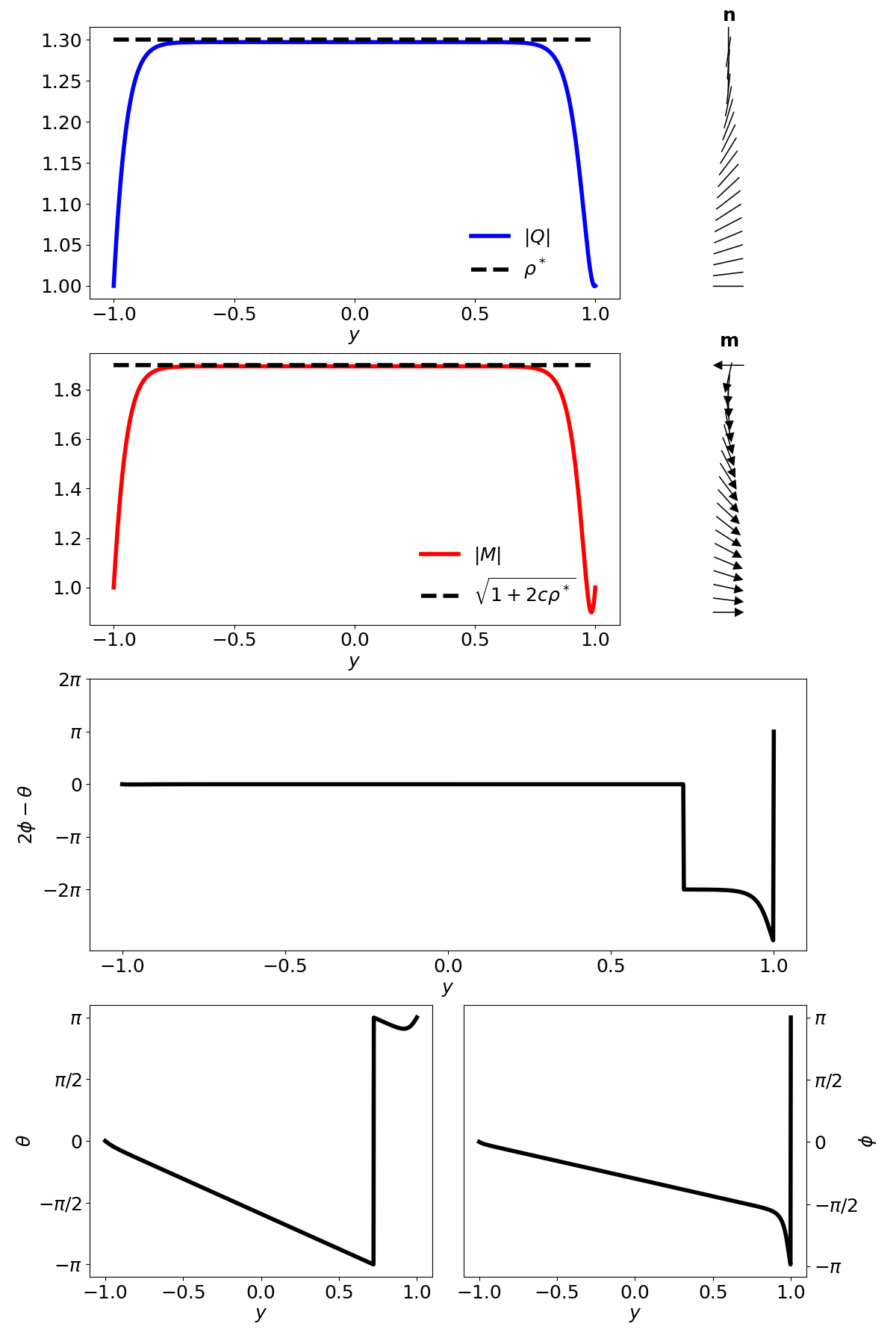}\\
        (Solution $2$; stable)
    \end{minipage}
    \begin{minipage}{0.49\textwidth}
        \centering
        \includegraphics[width=0.80\textwidth]{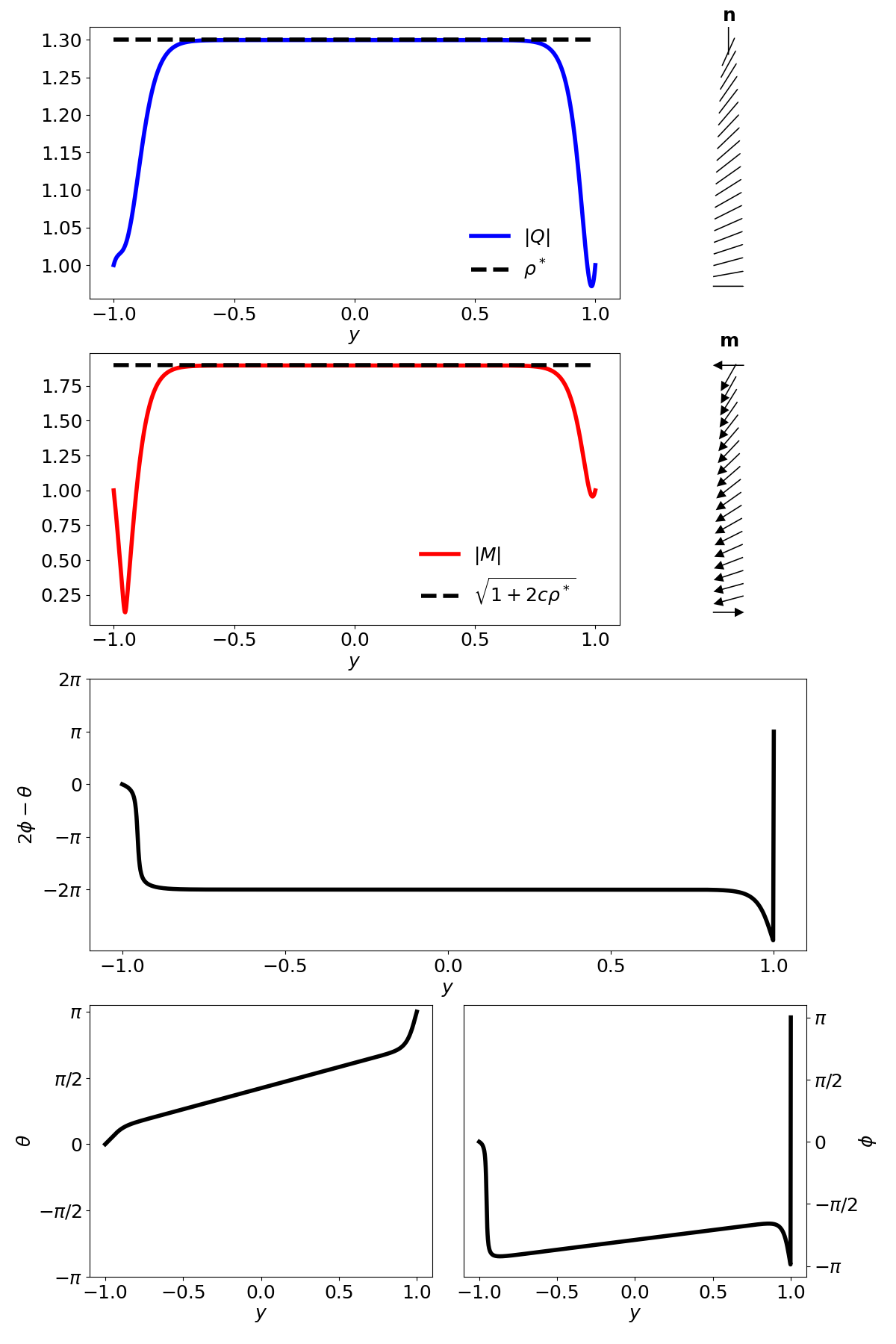}\\
        (Solution $3$; stable)
    \end{minipage}
    \begin{minipage}{0.49\textwidth}
        \centering
        \includegraphics[width=0.80\textwidth]{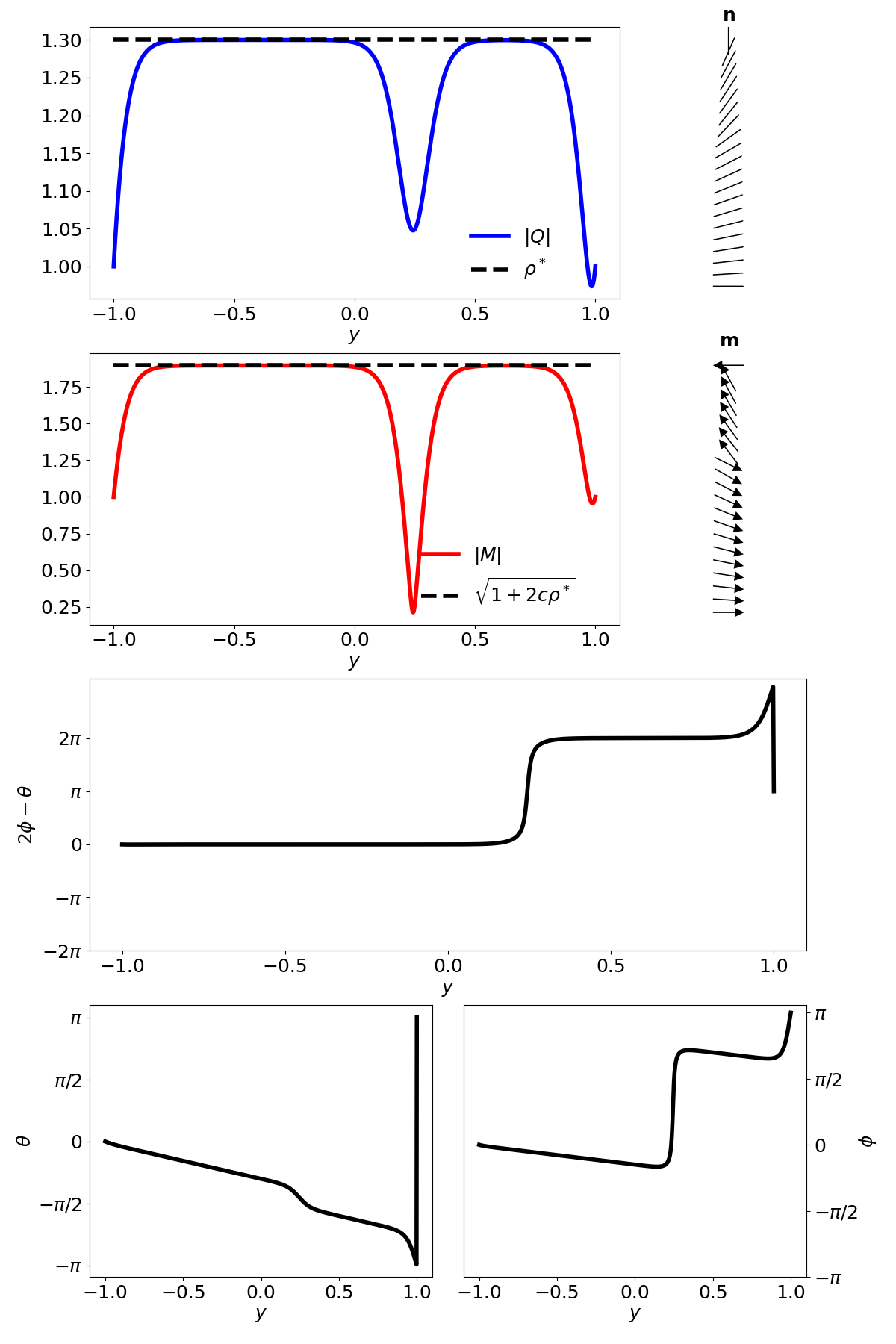}\\
        (Solution $4$; stable)
    \end{minipage}
    \caption{Four stable stationary profiles, $(Q_{11}, Q_{12}, M_1, M_2)$, of \cref{eq:4} with $l=0.01$ and $c=\xi=1$, along with plots of $(2\phi-\theta),\;\theta$ and $\phi$ to verify the relation \cref{eq:relation-theta-phi}.
Solutions $1$ and $2$ have the lowest full energy value \cref{eq:4}.}
    \label{fig:c1-full}
\end{figure}

Now, we repeat the simulations for $c=5$.
Two stable stationary profiles are illustrated in \cref{fig:c5-full}.
We see that $|\Qvec| \to \rho^*$ and $|\Mvec|^2 \to 1 + 2c\rho^*$ almost everywhere, as expected.
Here, Solution $2$ has lower energy than Solution $1$, since Solution $1$ has more local minima in $|\Qvec|$ and $|\Mvec|$ than Solution $2$.
Further, $(2\phi - \theta)$ is an even multiple of $\pi$ almost everywhere, with the jumps being associated with the local minima in $|\Qvec|$ and $|\Mvec|$, thus verifying  \cref{eq:constraint-theta-phi}.
Additionally, we plot $\phi$ and $\theta$ in \cref{fig:c5-full}, and observe almost linear profiles, except around the local minima and boundary layers. To summarise, the numerical experiments and the heuristics in Section~\ref{sec:convergence_full_problem} suggest that there are at least two energy minimisers, characterised by $(\rho_1, \sigma_1, \theta_1, \phi_1)$ and $(\rho_2, \sigma_2, \theta_2, \phi_2)$ of (\ref{eq:4}) in the $l \to 0$ limit,  such that $\rho_1, \rho_2 \to \rho^*$, $\sigma_{1,2}^2 \to 1 + 2c\rho^*$ almost everywhere away from $y=\pm 1$, $\theta_2 = -\theta_1$, $\phi_2 = -\phi_1$, with no domain walls and $2\phi_{1,2} - \theta_{1,2} $ an even multiple of $\pi$ except near $y=1$. The two energy minimisers differ in their sense of rotation, in $\nvec$ and $\mvec$, between $y=\pm 1$.

\begin{figure}[!ht]
    \centering
    \begin{minipage}{0.49\textwidth}
        \centering
        \includegraphics[width=0.80\textwidth]{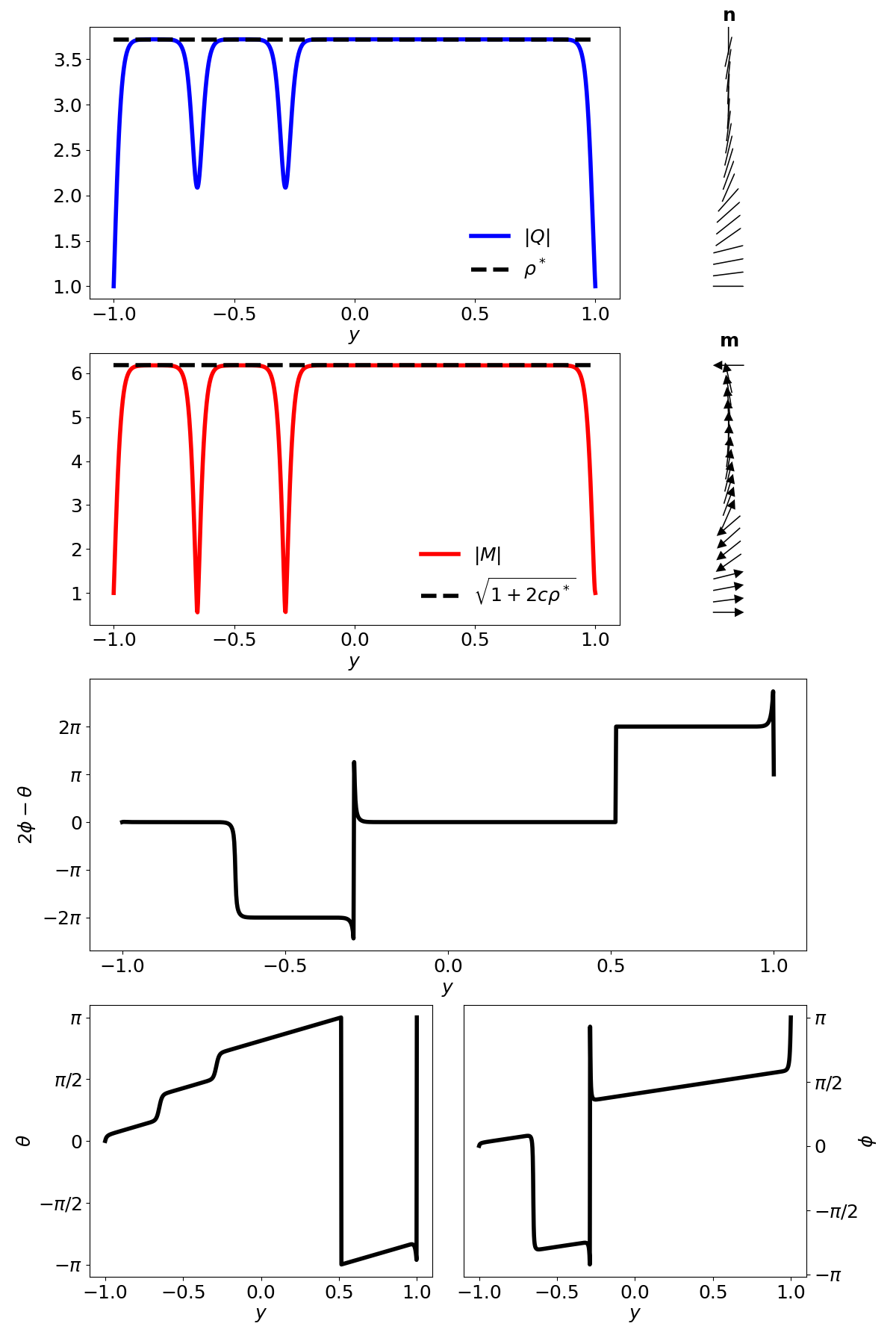}\\
        (Solution $1$; stable)
    \end{minipage}
    \begin{minipage}{0.49\textwidth}
        \centering
        \includegraphics[width=0.80\textwidth]{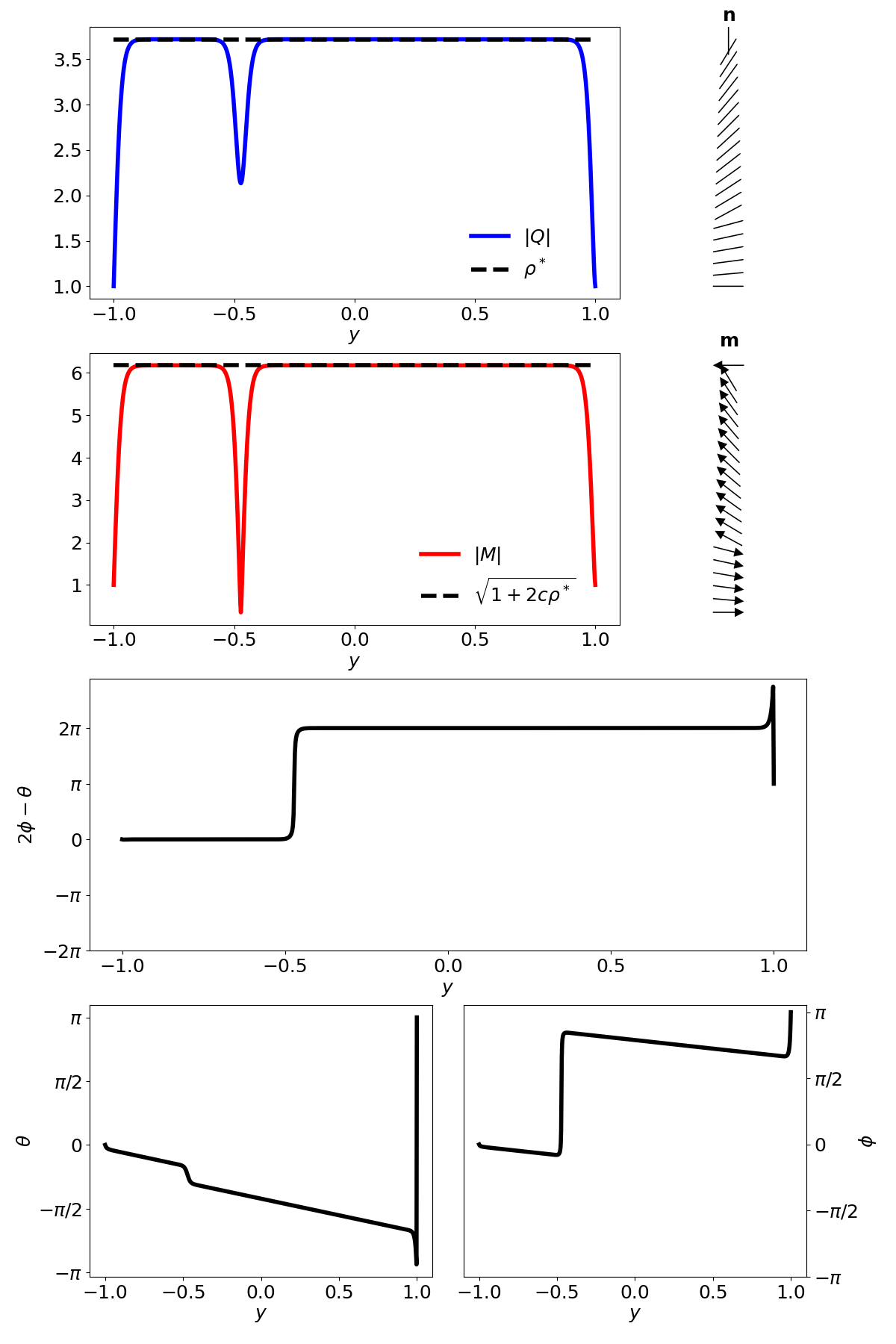}\\
        (Solution $2$; stable)
    \end{minipage}
    \caption{Two examples of stable stationary profiles $(Q_{11}, Q_{12}, M_1, M_2)$ of the full energy \cref{eq:4} with $l=0.01$, $c=5$ and $\xi=1$.
Solution $2$ has lower energy than Solution $1$.}
    \label{fig:c5-full}
\end{figure}

\subsection{Bifurcation diagram with continuing $l$}
We vary $l_1=l_2=l\in [0.2,3.0]$ with step size $0.01$ and $c=1$ in \cref{fig:lcont}.
There is only one stable OR solution for $l\in [1.25, 3.0]$, being the energy minimiser of the full energy \cref{eq:4}.
For $l\approx 1.25$, there is a pitchfork bifurcation consisting of two stable solution branches and one unstable OR branch (also see \cref{fig:l1}). 
In fact, the two stable solutions (Solutions $1$ and $3$ in \cref{fig:l1}) differ by the sign of $Q_{12}$ and $M_2$, i.e., for every solution branch, $(Q_{11}, Q_{12}, M_1, M_2)$, there exists another solution branch with $(Q_{11}, - Q_{12}, M_1, -M_2 )$.
The stable solution branches correspond to a smooth rotation in $\nvec$, between $y=\pm 1$ and are actually the global energy minimisers for $l\le 1.25$.
\begin{figure}[!ht]
    \centering
    \begin{minipage}{0.49\textwidth}
        \centering
        \includegraphics[width=0.8\textwidth]{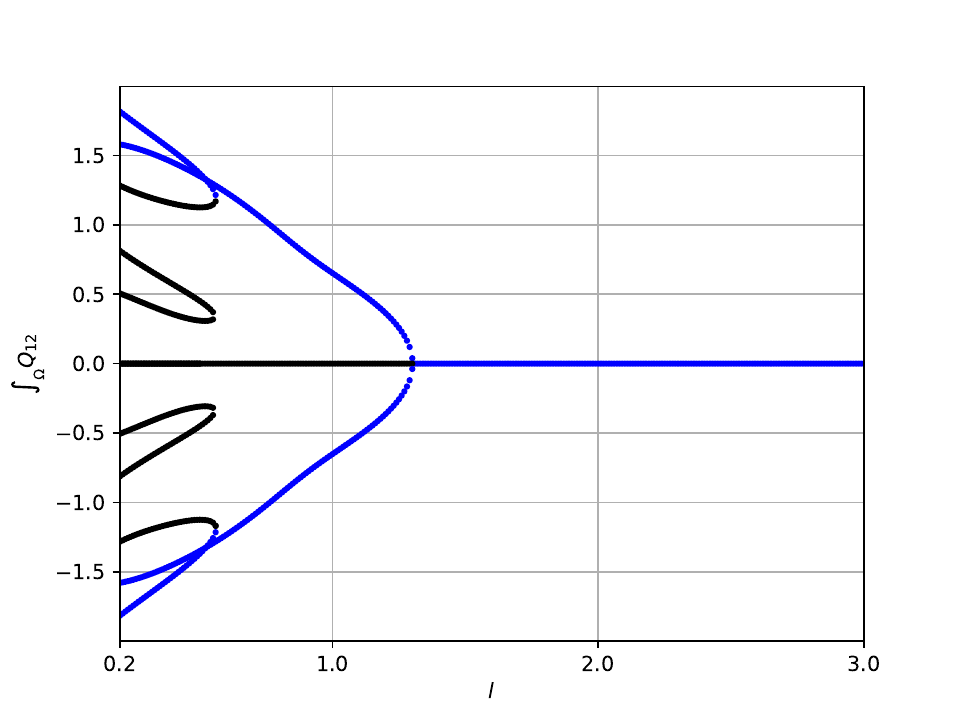}
    \end{minipage}
    \begin{minipage}{0.49\textwidth}
        \centering
        \includegraphics[width=0.8\textwidth]{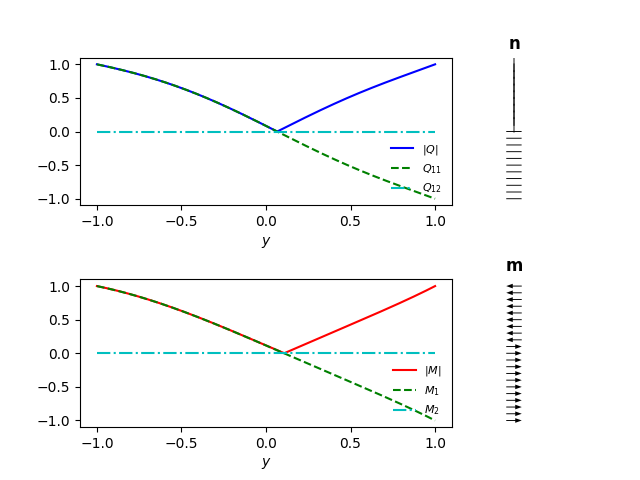}
    \end{minipage}
    \caption{Left: the bifurcation diagram of continuing $l_1=l_2=l\in[0.2,3.0]$ with fixed $c=\xi=1$; here, black represents unstable solutions while blue indicates stable solutions.
        Right: the stable solution for $l=2$.}
    \label{fig:lcont}
\end{figure}

\begin{figure}[!ht]
    \centering
    \begin{minipage}{0.32\textwidth}
        \centering
        \includegraphics[width=1.0\textwidth]{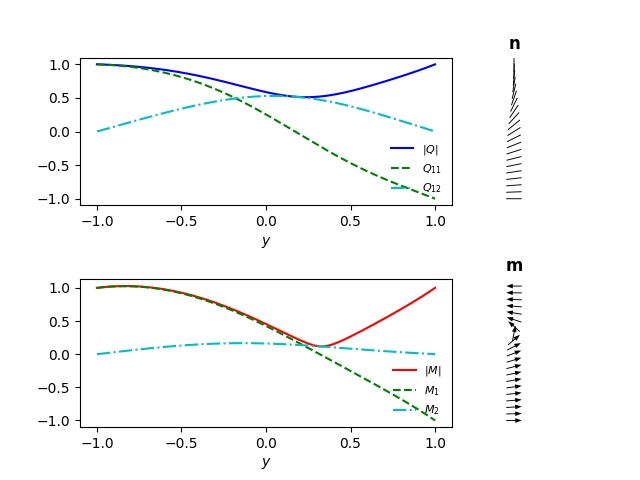}\\
        (Solution $1$; stable)
    \end{minipage}
    \begin{minipage}{0.32\textwidth}
        \centering
        \includegraphics[width=1.0\textwidth]{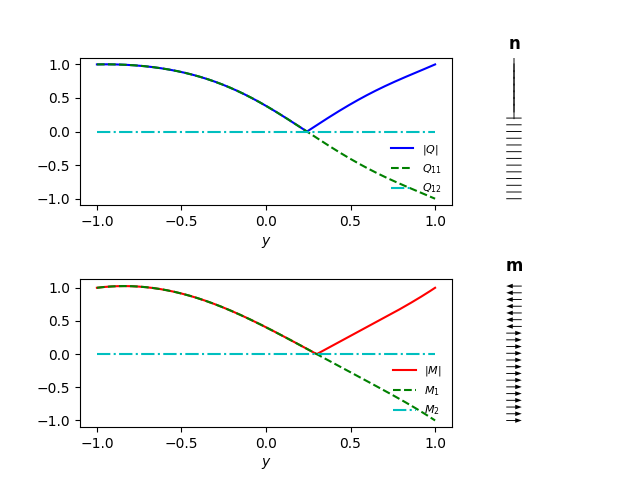}\\
        (Solution $2$; unstable)
    \end{minipage}
    \begin{minipage}{0.32\textwidth}
        \centering
        \includegraphics[width=1.0\textwidth]{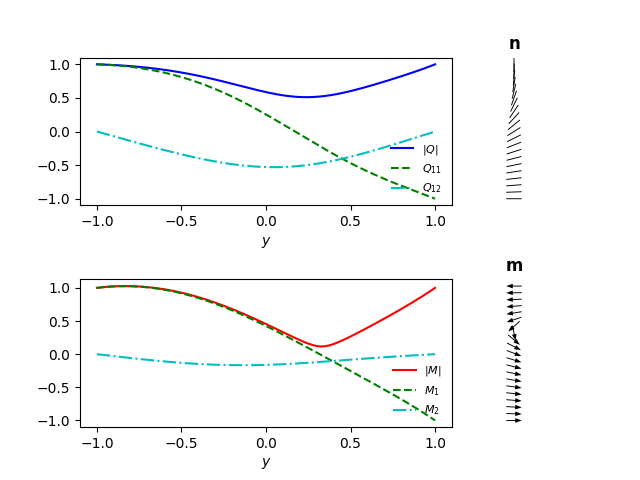}\\
        (Solution $3$; stable)
    \end{minipage}
    \caption{Three solutions for $l=1$ in \cref{fig:lcont}.
    Solutions $1$ and $3$ are global energy minimisers.}
    \label{fig:l1}
\end{figure}

As $l$ becomes smaller, more (stable or unstable) solutions are found.
More specifically, there are four disconnected bifurcations appearing around $l=0.55$, giving two further stable solutions, which are also local energy minimisers (see \cref{fig:l0.2} for an illustration) for $l\in [0.2,0.55]$.
Again, they only differ by the sign of $Q_{12}$ and $M_2$.
In \cref{fig:l0.2}, we plot two examples of newly found stable solution profiles.
The stable solutions typically correspond to a smooth $\nvec$-profiles with minimal rotation (minimal topological degree consistent with the boundary conditions),
while the stable normalised magnetisation profiles $\mvec$ are also smooth, except near $y=\pm 1$.

\begin{figure}[!ht]
    \centering
    \begin{minipage}{0.32\textwidth}
        \centering
        \includegraphics[width=1.0\textwidth]{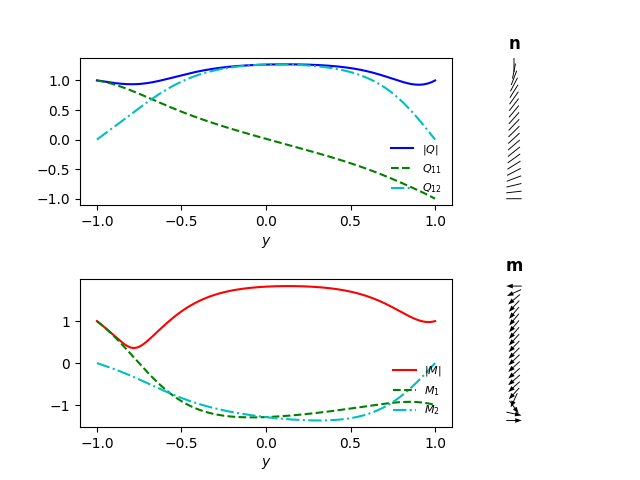}\\
        (Solution $1$; stable)
    \end{minipage}
    \begin{minipage}{0.32\textwidth}
        \centering
        \includegraphics[width=1.0\textwidth]{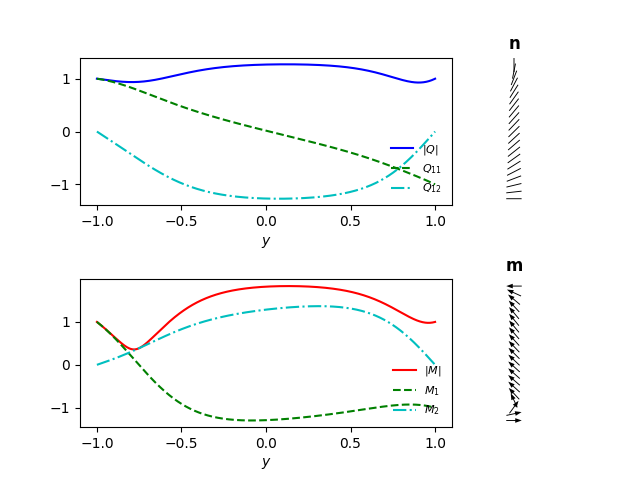}\\
        (Solution $4$; stable)
    \end{minipage}
    \caption{Two examples of new stable solutions for $l=0.2$ in \cref{fig:lcont}.
    They are global energy minimisers.}
    \label{fig:l0.2}
\end{figure}

We next consider the case of $c=5$, by numerically computing a bifurcation diagram  in \cref{fig:bif-c5}, for the solutions of \cref{eq:Q11}-\cref{eq:M2}, by continuing $l\in [3,5]$ with a step size of $0.015$. 
The globally stable OR solution is shown in \cref{fig:bif-c5} and it loses stability at the pitchfork bifurcation point $l\approx 4.44$, leading to two new stable branches (see illustrations in \cref{fig:l443} for $l=4.43$).
The new stable solutions only differ in the signs of $Q_{12}$ and $M_2$ and are in fact, energy minimisers for $l\le 4.34$.
Thus, the qualitative features of the bifurcation diagram are unchanged by increasing $c$, but the OR solution branch loses stability for $l < l^*(c)$, where $l^*(c) $ is an increasing function of $c$. Hence, as $c$ increases, OR solutions are increasingly difficult to find owing to their shrinking window of stability.

\begin{remark}
    We comment on the two folds in the bifurcation diagram \cref{fig:bif-c5}.
    They do not represent the same solution branch at the intersection points.
    Instead, they are just overlapping points in this plot of $\int_\Omega Q_{12}$ versus $l$. A different functional may yield a bifurcation diagram without these intersection points.
\end{remark}

\begin{figure}[!ht]
    \centering
    \begin{minipage}{0.49\textwidth}
        \centering
        \includegraphics[width=0.8\textwidth]{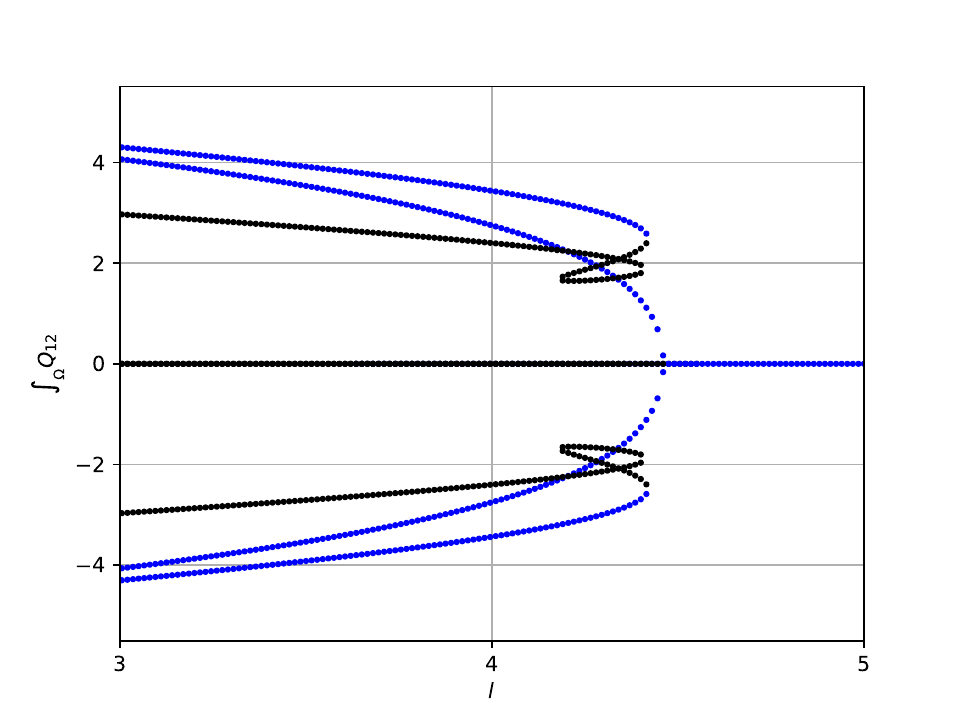}
    \end{minipage}
    \begin{minipage}{0.49\textwidth}
        \centering
        \includegraphics[width=0.8\textwidth]{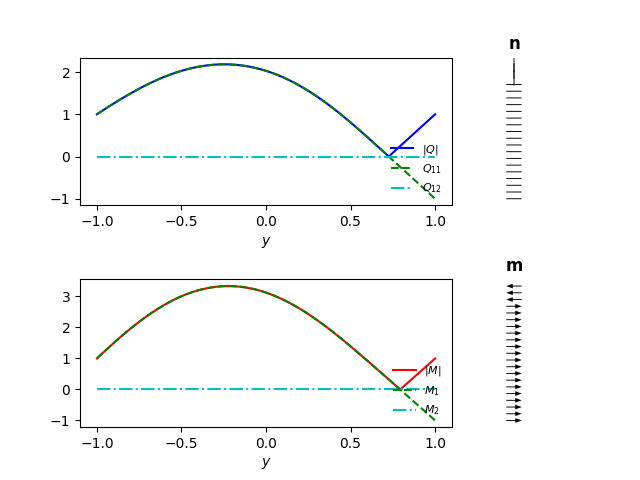}
    \end{minipage}
    \caption{Left: the bifurcation diagram with fixed $c=5$ and $\xi=1$; here, black labels unstable solutions while blue labels stable solutions.
    Right: one stable OR solution for $l=4.45$.}
    \label{fig:bif-c5}
\end{figure}

\begin{figure}[!ht]
    \centering
    \begin{minipage}{0.40\textwidth}
        \centering
        \includegraphics[width=1.0\textwidth]{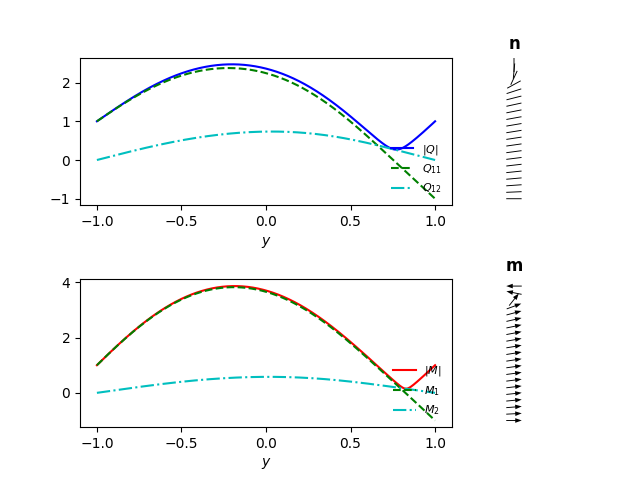}
    \end{minipage}
    \begin{minipage}{0.40\textwidth}
        \centering
        \includegraphics[width=1.0\textwidth]{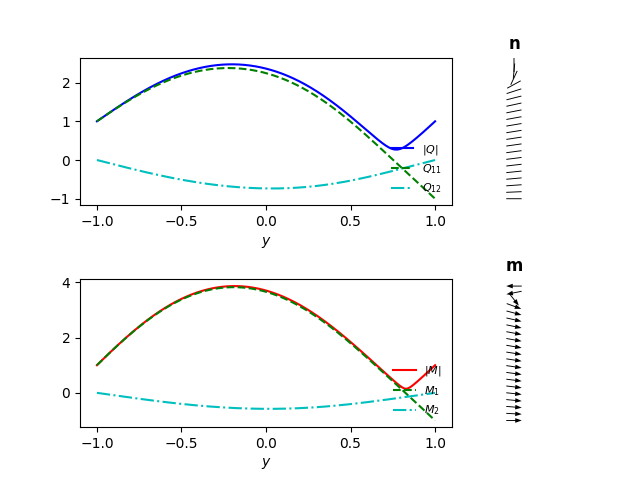}
    \end{minipage}
    \caption{Two new stable solutions at $l=4.43$ in \cref{fig:bif-c5}.}
    \label{fig:l443}
\end{figure}

\section{Conclusions}
\label{sec:conclusions}
We study confined systems with both nematic and magnetic order, inside a channel geometry with Dirichlet boundary conditions.
Specifically, we model the stable equilibria as minimisers of an appropriately defined energy on an interval $\left[-D, D \right]$, with three contributions: a nematic energy, a magnetic energy and a nemato-magnetic coupling energy.
We are interested in two parameters: the scaled elastic parameter $l$ that is inversely proportional to $D^2$, and the nemato-magnetic coupling parameter $c$.
We rigorously show that $c$ reduces the effective nematic correlation length $\xi_n$, for large $c$, and  we have the unique OR solution for $D \ll \frac{c^*}{c} \xi_n$, for some explicitly computable constant $c^*$, independent of $c$. 
The OR solution necessarily has separate nematic and magnetic domain walls, which are surface defects. 
As $D$ increases for fixed $c$ (or $c$ increases for fixed $D$), there can be multiple OR solutions, all of which are unstable with varying locations and multiplicities of domain walls, and the stable solutions do not have domain walls or polydomains for large $D$.
There are multiple stable solutions for large $D$, characterised by the rotation profiles of $\nvec$ and $\mvec$ between the boundaries.
Our choice of boundary conditions necessarily lead to boundary layers, which again will have distinct optical signatures, if implemented.
We have provided analytic characterisations of the limiting profiles for small $D$ (in terms of the OR solution) and large $D$ (in terms of limiting maps) accompanied by extensive numerical studies, which beautifully illustrate how we can use $l$ and $c$ to tune domain walls, boundary layers and multistability, all of which can be exploited for optical and mechanical responses. 
This work gives informative insight into the complex interplay between geometry, material properties, temperature (captured by $l$), nemato-magnetic coupling and boundary conditions in the solution landscapes (also see \cite{maity2021} for the numerical analysis of this system). 
Our methods can be modified to include different types of boundary conditions and nemato-magnetic coupling, which could enhance the stability of OR solutions, and we will develop universal theoretical frameworks for composite materials with multiple order parameters in future work.  




\bibliographystyle{siamplain}
\bibliography{main}
\end{document}


\maketitle

\begin{keywords}
  ferronematics, bifurcation analysis, stability, liquid crystals
\end{keywords}

\begin{AMS}
  34D20, 34C23, 76A15
\end{AMS}

\section{Full problem: homogeneous solutions and asymptotics as $c\to 0$ and $c\to \infty$}
\label{sec:homo-full}

The critical points of the bulk energy density $f(Q_{11}, Q_{12}, M_1, M_2)$, particularly the bulk energy minimisers for different values of $(c, \xi)$, are crucial for the analysis of the full variational problem in the main text.
We denote the bulk energy density by
\begin{equation}
    \label{eq:bulk_energy_density}
    \begin{aligned}
    f(Q_{11}, Q_{12}, M_1, M_2)
    \coloneqq  &\left( Q_{11}^2 + Q_{12}^2 - 1 \right)^2 + \frac{\xi}{4}\left(M_1^2 + M_2^2 - 1 \right)^2\\
               &- cQ_{11}\left(M_1^2 - M_2^2 \right) - 2c Q_{12}M_1 M_2.
    \end{aligned}
\end{equation}
Using the substitutions
\begin{equation}
    \label{eq:substitution}
    \begin{aligned}
        &Q_{11}=\rho\cos(\theta), Q_{12}=\rho\sin(\theta),\\
        &M_1=\sigma\cos(\phi), M_2=\sigma\sin(\phi),
    \end{aligned}
\end{equation}
the bulk energy density becomes
\begin{equation}
\label{eq:f}
    f(\rho,\sigma,\theta,\phi)=(\rho^2-1)^2+\frac{\xi}{4}(\sigma^2-1)^2-c\rho\sigma^2\cos(2\phi-\theta).
\end{equation}
The critical points $(\rho (c, \xi), \sigma(c, \xi))$, $\rho,\sigma\geq 0$ are solutions of the following system of algebraic equations:
\begin{subequations}
    \label{eq:substitute-system}
    \begin{align}
        &4\rho\cos(\theta)\left(\rho^2-1\right)-c\sigma^2\cos(2\phi)=0, \label{eq:sub1}\\
        &4\rho\sin(\theta)\left(\rho^2-1\right)-c\sigma^2\sin(2\phi)=0, \label{eq:sub2}\\
        &\xi\sigma\cos(\phi)\left(\sigma^2-1\right)-2\sigma\rho c\cos(\theta-\phi)=0, \label{eq:sub3}\\
        &\xi\sigma\sin(\phi)\left(\sigma^2-1\right)-2\sigma\rho c\sin(\theta-\phi)=0.\label{eq:sub4}
    \end{align}
\end{subequations}
There are two trivial solutions, i.e., $\rho=\sigma=0$; $\rho=1$ and $\sigma=0$, which exist for all real-valued $c$ and $\xi$. For these trivial solutions, $\theta$ and $\phi$ are undetermined.
For ferronematic solutions, we need $\rho,\sigma\ne 0$ in order to capture the nemato-magnetic coupling. The equations \cref{eq:substitute-system} can be explicitly solved;  firstly, \cref{eq:sub3} and \cref{eq:sub4} give
\begin{equation*}
    	\sigma^2=1+\frac{2\rho c\cos(\theta-\phi)}{\xi\cos(\phi)} \textrm{ and }
    	\sigma^2=1+\frac{2\rho c\sin(\theta-\phi)}{\xi\sin(\phi)},
\end{equation*}
respectively, which in turn require that 
\begin{equation*}
    \frac{\cos(\theta-\phi)}{\cos(\phi)}=\frac{\sin(\theta-\phi)}{\sin(\phi)}\implies 2\phi-\theta=n\pi\quad\textrm{for }n\in\mathbb{Z},
\end{equation*} imposing constraints on the relative alignment of $\nvec$ and $\Mvec$.
Furthermore, multiplying \cref{eq:sub1} by $\sin(\theta)$, \cref{eq:sub2} by $\cos(\theta)$,  following similar steps for \cref{eq:sub3} and \cref{eq:sub4}, and using the constraint $2\phi=\theta+n\pi$, we obtain
\begin{subequations}
    \label{eq:cubic}
    \begin{align}
&	\rho^3-\rho\left(1+\frac{c^2}{2\xi}\right)-\frac{c}{4}=0 \label{eq:cubic1}, \\
&	\rho^3-\rho\left(1+\frac{c^2}{2\xi}\right)+\frac{c}{4}=0 \label{eq:cubic2}.
\end{align}
\end{subequations}
Here, \cref{eq:cubic1} corresponds to $2\phi-\theta = 2n \pi$ whereas \cref{eq:cubic2} corresponds to $2\phi-\theta = (2n + 1) \pi$, i.e., an odd multiple of $\pi$. Once $\rho$ is determined by \cref{eq:cubic}, $\sigma$ is given by
\begin{subequations}
\begin{align}
&    \sigma^2=1+\frac{2\rho c}{\xi} \label{eq:sigma_cubic1},\\
&    \sigma^2=1-\frac{2\rho c}{\xi} \label{eq:sigma_cubic2};
    \end{align}
\end{subequations}
where \cref{eq:sigma_cubic1} (resp.\ \cref{eq:sigma_cubic2}) is associated with \cref{eq:cubic1} (resp.\ \cref{eq:cubic2}).

To solve \cref{eq:cubic1}, let $\rho=S+T$ and recall that $(S+T)^3-3ST(S+T)-(S^3+T^3)=0$, then we have
\begin{eqnarray*}
&& ST=\frac{1}{3}\left(1+\frac{c^2}{2\xi}\right), \qquad 
 S^3+T^3=\frac{c}{4},
\end{eqnarray*}
from which we deduce that $S^3$ and $T^3$ are roots of the quadratic equation
\begin{equation*}
	z^2-\frac{c}{4}z+\frac{1}{27}\left(1+\frac{c^2}{2\xi}\right)^3=0,
\end{equation*}
i.e., $z_\pm=\frac{c}{8}\pm\sqrt{\frac{c^2}{64}-\frac{1}{27}\left(1+\frac{c^2}{2\xi}\right)^3}$.
Hence, $S$ and $T$ are given by
\begin{subequations}
    \begin{align}
    & S=\omega_j\left(\frac{c}{8}+\sqrt{\frac{c^2}{64}-\frac{1}{27}\left(1+\frac{c^2}{2\xi}\right)^3}\right)^\frac{1}{3} \eqqcolon \omega_j \Theta_1 \label{eq:S_def},\\
    & T=\omega_k\left(\frac{c}{8}-\sqrt{\frac{c^2}{64}-\frac{1}{27}\left(1+\frac{c^2}{2\xi}\right)^3}\right)^\frac{1}{3} \eqqcolon \omega_k \Theta_2,\label{eq:S_T_def}
    \end{align}
\end{subequations}
for $j,k\in\{1,2,3\}$, where $\omega_1=1, \omega_2=\frac{-1+\sqrt{3}\mi}{2}$ and $\omega_3=\frac{-1-\sqrt{3}\mi}{2}$ are the cube roots of unity.
Throughout this work, we use the notation $\mi=\sqrt{-1}$.
This yields nine possible roots of \cref{eq:cubic1},
of which only three are viable choices since $ST=\frac{1}{3}\left(1+\frac{c^2}{2\xi}\right)$ and thus, $\omega_j\omega_k=1$.
Therefore, the viable roots of \cref{eq:cubic1} are given by, for $k\in \{1,2,3\}$,
\begin{equation}
    \rho=
    \omega_k \Theta_1 + \omega_k^2 \Theta_2
    \label{eq:rho_even},
\end{equation}
and the corresponding values of $\sigma$ are
\begin{equation}
    \label{eq:sigma_even}
    \sigma=
    \sqrt{1+\frac{2c}{\xi} \left( \omega_k \Theta_1 + \omega_k^2 \Theta_2 \right)}
\end{equation}
with $2\phi-\theta$ being an even multiple of $\pi$.

The cubic equation \cref{eq:cubic2} can be tackled similarly so that the relevant roots of \cref{eq:cubic2} and the corresponding values of $\sigma$ are given by
\begin{align}
    \rho &=\omega_k\left(-\frac{c}{8}+\sqrt{\frac{c^2}{64}-\frac{1}{27}\left(1+\frac{c^2}{2\xi}\right)^3}\right)^\frac{1}{3}+\omega_k^2\left(-\frac{c}{8}-\sqrt{\frac{c^2}{64}-\frac{1}{27}\left(1+\frac{c^2}{2\xi}\right)^3}\right)^\frac{1}{3} \label{eq:rho_odd} \\
         &\eqqcolon \omega_k\Lambda_1 + \omega_k^2 \Lambda_2, \nonumber 
\end{align}
and
\begin{equation}
\label{eq:sigma_odd}
\sigma=
\sqrt{1-\frac{2c}{\xi} \left( \omega_k\Lambda_1 + \omega_k^2 \Lambda_2 \right)},
\end{equation}
with $2\phi - \theta$ being an odd multiple of $\pi$, for $k\in\{1,2,3\}$.
Hence, the solution pairs $(\rho (c, \xi),\sigma(c, \xi))$ for the algebraic system \cref{eq:substitute-system} are given by (\cref{eq:rho_even},\cref{eq:sigma_even}) and (\cref{eq:rho_odd},\cref{eq:sigma_odd}).

Amongst the admissible critical points, we require $\rho$ and $\sigma$ to be non-negative and real-valued. In fact, if $27c^2\leq 64\left(1+\frac{c^2}{2\xi}\right)$ then both \cref{eq:rho_even} and \cref{eq:rho_odd} yield a real-valued $\rho$ for each $k\in\{1,2,3\}$ (one can check this with De Moivre's formula), whilst for $27c^2>64\left(1+\frac{c^2}{2\xi}\right)$, \cref{eq:rho_even} and \cref{eq:rho_odd} have a real solution only if $k=1$.
Furthermore, if $\rho$ is real, \cref{eq:sigma_even} gives a real-valued $\sigma$ when $-\frac{2c}{\xi}\rho\leq 1$, whilst \cref{eq:sigma_odd} yields a real-valued $\sigma$ if $\frac{2c}{\xi}\rho\leq 1$.
Regarding non-negative solution pairs, we simply enumerate the solution pairs with $\rho\ge 0$ and choose the positive root for $\sigma$.

For simplicity, we now focus on the special case of $\xi=1$.
The inequality $27c^2<64\left(1+\frac{c^2}{2}\right)^3$ holds for all $c\geq 0$, and hence, 
both \cref{eq:rho_even} and \cref{eq:rho_odd} yield real-valued $\rho$.
Therefore, we have the following positive real solutions:%
\begin{subequations}
    \begin{align}
    &\rho= \Theta_1 +\Theta_2 \label{eq:rho1},\\
    &\rho= \Lambda_1 + \Lambda_2\label{eq:rho2},\\
    &\rho= \omega_3 \Lambda_1 + \omega_2 \Lambda_2\label{eq:rho3},
    \end{align}
\end{subequations}
and the corresponding values of $\sigma$ are
\begin{subequations}
    \begin{align}
    &\sigma= \sqrt{1+2c\left(\Theta_1 +\Theta_2 \right)},
    \label{eq:sigma1}\\
    &\sigma= \sqrt{1-2c\left( \Lambda_1 + \Lambda_2 \right)},
    \label{eq:sigma2}\\
    &\sigma= \sqrt{1-2c \left( \omega_3 \Lambda_1 + \omega_2 \Lambda_2 \right)}.
    \label{eq:sigma3}
    \end{align}
\end{subequations}
Moreover, the solutions in  \cref{eq:sigma1} and \cref{eq:sigma3} correspond to a real-valued $\sigma$ for all $c>0$, while \cref{eq:sigma2} is real for $c\leq \frac{1}{2}$. Henceforth, we do not consider the solution pair (\cref{eq:rho2},\cref{eq:sigma2}) in this manuscript.

The next task is to determine which solution pair, $(\rho,\sigma)$, minimises the bulk energy density \cref{eq:f}, for $c\geq 0$ with $\xi=1$.
It is enough to observe that $f(0,0)=\frac{5}{4}$ and $f(1,0)=\frac{1}{4}$, and identify which of the non-trivial solution pairs, if any, satisfy $f(\rho,\sigma)<\frac{1}{4}$.
In \cref{fig:energy_plots}, we plot the bulk energy density \cref{eq:f} as a function of $c$ for each non-trivial solution pair (\cref{eq:rho1}, \cref{eq:sigma1}), (\cref{eq:rho3}, \cref{eq:sigma3}) and observe that (\cref{eq:rho1},\cref{eq:sigma1}) is the global energy minimiser for all positive values of $c$.
\begin{figure}[!ht]
    \centering
    \begin{minipage}{0.49\textwidth}
        \centering
        \includegraphics[scale=0.34]{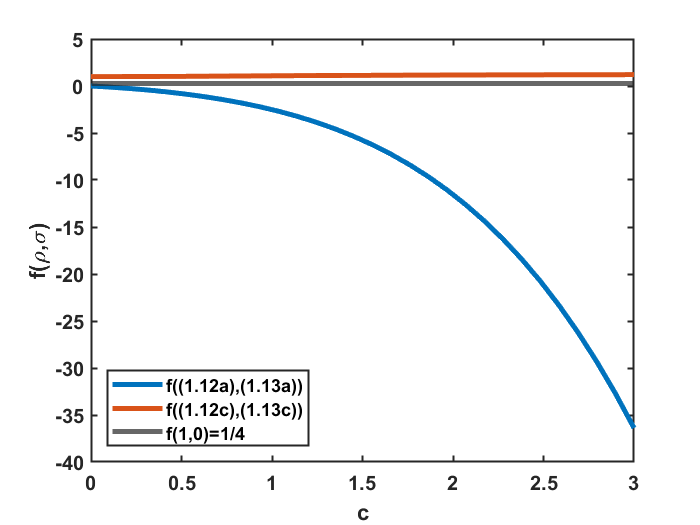}
    \end{minipage}
    \begin{minipage}{0.49\textwidth}
        \centering
        \includegraphics[scale=0.34]{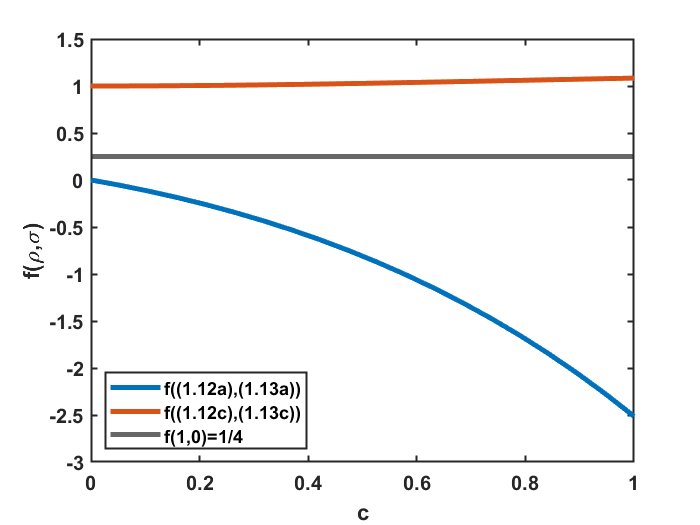}
    \end{minipage}
    \caption{Plots of the bulk energy density \cref{eq:f} at $(\rho,\sigma)=$ (\cref{eq:rho1},\cref{eq:sigma1}) and (\cref{eq:rho3},\cref{eq:sigma3}) as functions of $c$ with $\xi=1$ for
        (left) $c\in[0,3]$ and (right) $c\in[0,1]$.
}
    \label{fig:energy_plots}
\end{figure}

Next, we compute asymptotic expansions for the minimising pair $(\cref{eq:rho1}, \cref{eq:sigma1})$ as $c\to 0$, $c\to \infty$ respectively. As $c\to 0$, we expect the minimising pair to approach $(\rho, \sigma) \to (1,1)$, since $(\rho, \sigma) = (1,1)$ is the minimiser of the bulk energy density with $c=0$.
To this end, we first approximate the square root terms in \cref{eq:rho1} as
\begin{equation*}
    \sqrt{-\frac{c^2}{64}+\frac{1}{27}\left(1+\frac{c^2}{2}\right)^3}\approx\frac{1}{3\sqrt{3}}\left(1+\frac{69c^2}{128}\right),
\end{equation*}
by 
applying the binomial formula.
Then using De Moivre's formula and an appropriate binomial formula, the first cube root is approximated by
\begin{align*}
    \Theta_1 = \left(\frac{c}{8}+\sqrt{\frac{c^2}{64}-\frac{1}{27}\left(1+\frac{c^2}{2}\right)^3}\right)^\frac{1}{3}
    &\approx \left(\frac{\mi}{3\sqrt{3}}\right)^\frac{1}{3}\left(1+\left(\frac{3\sqrt{3}c}{8\mi}+\frac{69c^2}{128}\right)\right)^\frac{1}{3}\\
    &=\left(\frac{1}{2}+\frac{\sqrt{3}\mi}{6}\right)\left(1+\frac{\sqrt{3}c}{8\mi}+\frac{29c^2}{128}+\mathcal{O}(c^3)\right)\\
    &\approx\frac{1}{2}+\frac{c}{16} + \mi P(c).
\end{align*}
In the last approximation, we have ignored $\mathcal{O}(c^2)$ terms for small $c$ and $P(c)\coloneqq \frac{\sqrt{3}}{6}-\frac{\sqrt{3}c}{16}$.
The second cube root is the complex conjugate of the first cube root in \cref{eq:rho1}, and thus for the solution pair (\cref{eq:rho1}, \cref{eq:sigma1}), we have as $c\to 0$ that
\begin{subequations}
    \label{eq:c-small-approx1}
    \begin{align}
    &	\rho\approx1+\frac{c}{8},\label{eq:rho1_c_small}\\
    &	\sigma^2\approx 1+2c+\frac{c^2}{4}\label{eq:sigma1_c_small}.
    \end{align}
\end{subequations}


For large $c \gg 1$, the square root in \cref{eq:rho1} is given by (to leading order),
\begin{equation}
    \label{eq:squareroot-approx}
\sqrt{\frac{1}{27}\left(1+\frac{c^2}{2}\right)^3-\frac{c^2}{64}}\approx\frac{\sqrt{6}c^3}{36}.
\end{equation} 
Similarly, the first cube root in \cref{eq:rho1} can be approximated to leading order as
\begin{equation*}
    \Theta_1
     \approx \left(\frac{c}{8}+\mi \frac{\sqrt{6}c^3}{36}\right)^{\frac{1}{3}}
    \approx\left(\frac{\sqrt{3}+\mi}{2}\right)\left(\frac{\sqrt{6}}{36}\right)^\frac{1}{3}c.
\end{equation*}
Applying similar arguments to the second cube root, we deduce that for $c\gg 1$
\begin{subequations}
    \begin{align}
        & \rho\approx \left(\frac{\sqrt{2}}{4}\right)^\frac{1}{3}c, \label{eq:c-large-rhoapprox}\\
            & \sigma^2\approx 
            1+\sqrt{2}c^2.
    \end{align}
\end{subequations}
Hence, both $\rho$ and $\sigma$ grow linearly with $c$ for large $c\gg 1$, for the bulk energy minimiser (\cref{eq:rho1}, \cref{eq:sigma1}).

\begin{remark}\label{rem:uncoupled_solutions}
    For the uncoupled problem with $c=0$, we have nine solution pairs $(\rho,\sigma)$ (with $\theta,\phi$ undetermined) of the system \cref{eq:substitute-system}, namely,
    \begin{equation}
        \label{eq:nine-sol-pairs}
        (0,0),\;(0,\pm 1),\; (\pm 1,0 ),\; (\pm 1, \pm 1),
    \end{equation}
    where all sign combinations are possible.
    In the limit $c\to 0$, we expect each of the solution pairs $(\cref{eq:rho_even},\cref{eq:sigma_even})$ and $(\cref{eq:rho_odd},\cref{eq:sigma_odd})$ to reduce to a solution of the uncoupled problem.
    In fact, $(\cref{eq:rho_even},\cref{eq:sigma_even})$ and $(\cref{eq:rho_odd},\cref{eq:sigma_odd})$ recover all of the solution pairs in \cref{eq:nine-sol-pairs}, except $(0,0)$ and $(\pm 1,0)$.
   However, these unrecovered solution pairs, $(\rho, \sigma) = \left\{ (0,0), (\pm 1, 0) \right\}$ are solutions of the system \cref{eq:substitute-system} for arbitrary $c\geq 0$, and thus not perturbed for $c>0$.
\end{remark}

\section{Asymptotic expansions as $l\to\infty$ and $c\to 0$}

In the $l\to \infty$ limit, we compute useful asymptotic expansions of the OR solution branch for large $l$ and small $c$, by setting $l=\frac{1}{c}$ in the Euler--Lagrange equations stated in the main text and expanding around $(\Qvec^\infty, \Mvec^\infty)$ as shown below:

\begin{equation*}
    Q_{11}(y)  = -y + cf_2(y) + c^2 f_3(y) + \mathcal{O}(c^3), \quad M_1 = -y + cf_2^*(y) + c^2 f_3^*(y) + \mathcal{O}(c^3).
\end{equation*}
Substituting the above into the derived Euler--Lagrange equations in the main text (with $l=\frac{1}{c}$) and equating powers of $c$, we solve the computed second order ordinary differential equations for $f_2,f_3,f_2^*,f_3^*$, subject to the boundary conditions $f_2(-1)=f_2(1)=f_3(-1)=f_3(1)=0$ and $f^*_2(-1)=f^*_2(1)=f^*_3(-1)=f^*_3(1)=0$.
This gives
\begin{equation}
    \label{eq:Q-exp}
    Q_{11}(y) = -y + c\left(-\frac{1}{5} y^5 +\frac{2}{3} y^3 - \frac{7}{15}y\right) + c^2 p(y) + \mathcal{O}(c^3),
\end{equation}
and 
\begin{equation}
    \label{eq:M-exp}
    M_{1}(y) = -y + c \left(-\frac{1}{20} y^5 +\frac{1}{6} y^3 -\frac{7}{60} y\right) + c^2 q(y) + \mathcal{O}(c^3),
\end{equation}
for $l=\frac{1}{c}$ and $l \gg 1$.
Here,
\begin{equation*}
    p(y) = -\frac{1}{30} y^9+\frac{22}{105}y^7 -\frac{31}{75}y^5 -\frac{1}{12}y^4
    +\frac{14}{45}y^3-\frac{233}{3150} y+\frac{1}{12},
\end{equation*}
and
\begin{equation*}
    q(y) =  -\frac{1}{480} y^9 + \frac{11}{840}y^7-\frac{31}{1200}y^5 - \frac{1}{6}y^4 +\frac{7}{360}y^3 - \frac{233}{50400} y + \frac{1}{6}.
\end{equation*}

We can further check the validity of these expansions, \cref{eq:Q-exp} and \cref{eq:M-exp}, numerically.
To this end, we compare $(\cdot)^{num}$ and $(\cdot)^{asymp}$ in the $L^\infty$-norm,
where $(\cdot)^{num}$ is the numerical solution and $(\cdot)^{asymp}$ corresponds to the asymptotic expansion, depending on the truncation of the expansions in \cref{eq:Q-exp} and \cref{eq:M-exp}.
For instance, a first order truncation yields
\begin{equation*}
    \begin{aligned}
        Q_{11}^{asymp} &= -y + c\left(-\frac{1}{5} y^5 +\frac{2}{3} y^3 - \frac{7}{15}y\right),\\
        M_1^{asymp} &= -y + c\left(-\frac{1}{20} y^5 +\frac{1}{6} y^3 -\frac{7}{60} y\right).
    \end{aligned}
\end{equation*}
The left hand column of \cref{fig:c-asymp} shows a first order convergence by truncating the expansions up to $\mathcal{O}(c^0)$,
whilst a first order truncation leads to a second order convergence as shown in the middle column of \cref{fig:c-asymp} and finally, in the right hand column, a truncation up to $\mathcal{O}(c^2)$ demonstrates an almost third order convergence with respect to $c$, for both $Q_{11}$ and $M_1$.

\begin{figure}[ht!]
    \centering
    \begin{subfigure}{0.32\textwidth}
        \centering
    \includegraphics[scale=0.29, trim={0cm 0cm 0cm 0cm}, clip]{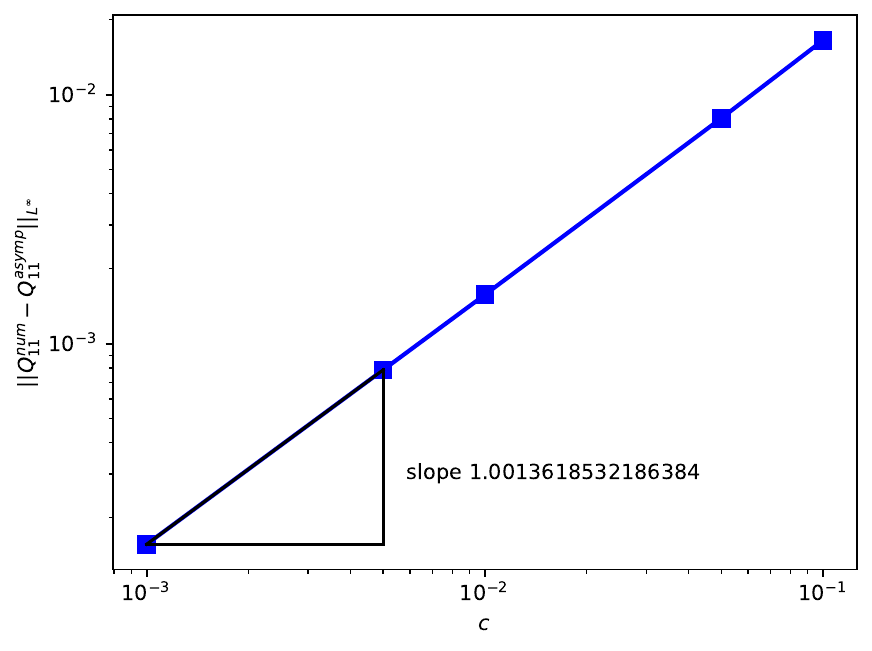}\\
\end{subfigure}\hfill
    \begin{subfigure}{0.32\textwidth}
        \centering
    \includegraphics[scale=0.29, trim={0cm 0cm 0cm 0cm}, clip]{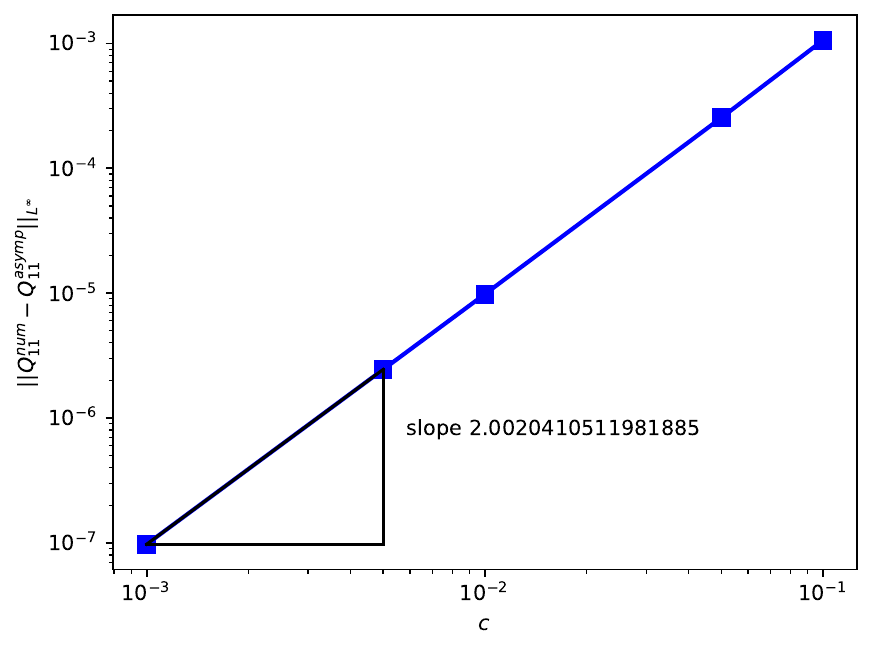}\\
\end{subfigure}\hfill
    \begin{subfigure}{0.32\textwidth}
        \centering
    \includegraphics[scale=0.29, trim={0cm 0cm 0cm 0cm}, clip]{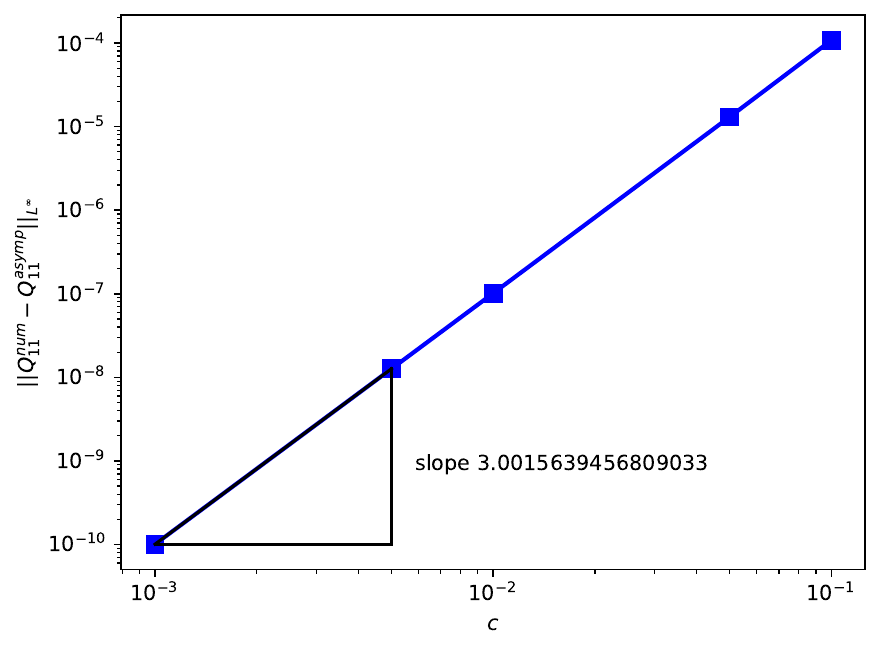}\\
\end{subfigure}\hfill
    \begin{subfigure}{0.32\textwidth}
        \centering
    \includegraphics[scale=0.29, trim={0cm 0cm 0cm 0cm}, clip]{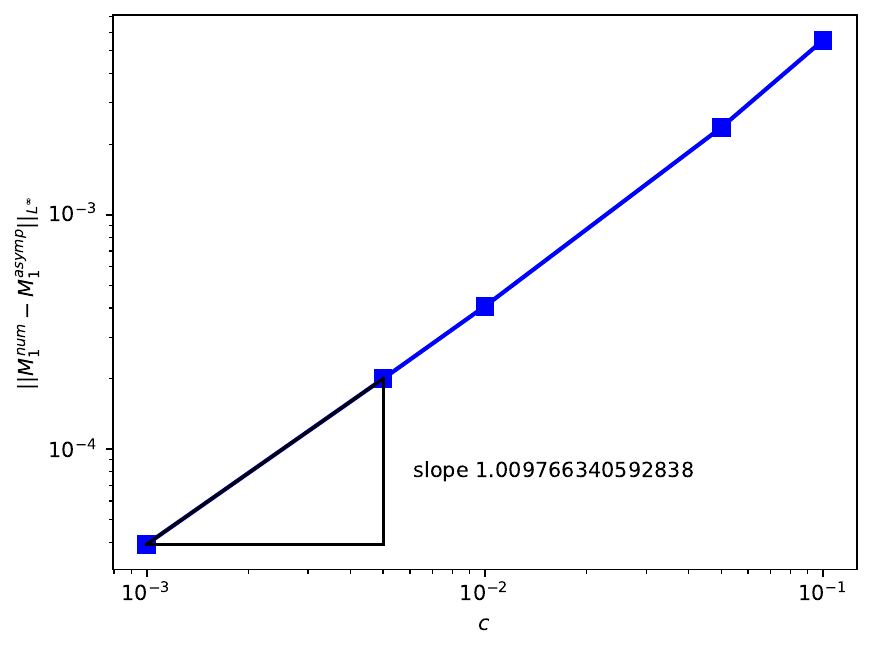}\\
\end{subfigure}\hfill
    \begin{subfigure}{0.32\textwidth}
        \centering
    \includegraphics[scale=0.29, trim={0cm 0cm 0cm 0cm}, clip]{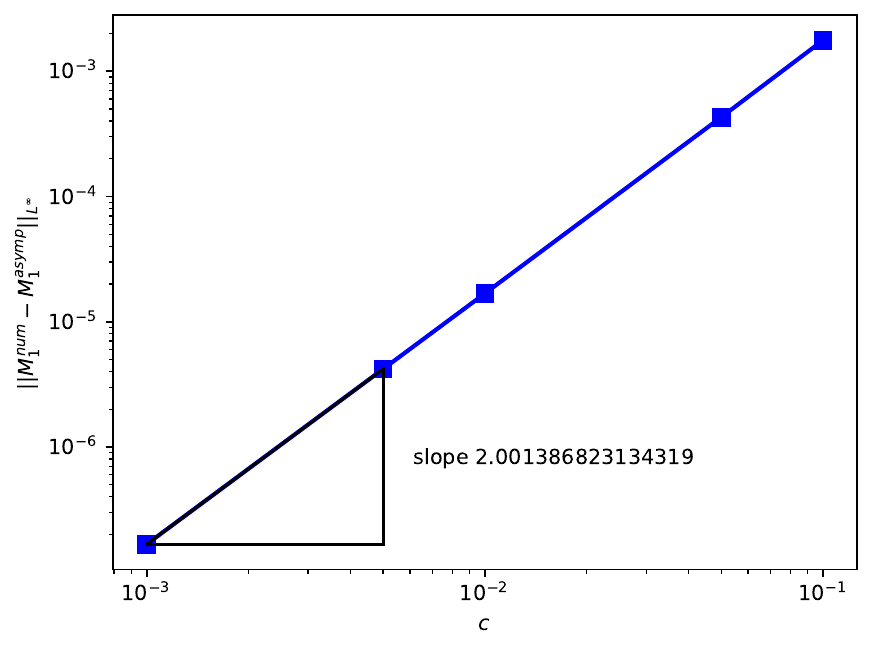}\\
\end{subfigure}\hfill
    \begin{subfigure}{0.32\textwidth}
        \centering
    \includegraphics[scale=0.29, trim={0cm 0cm 0cm 0cm}, clip]{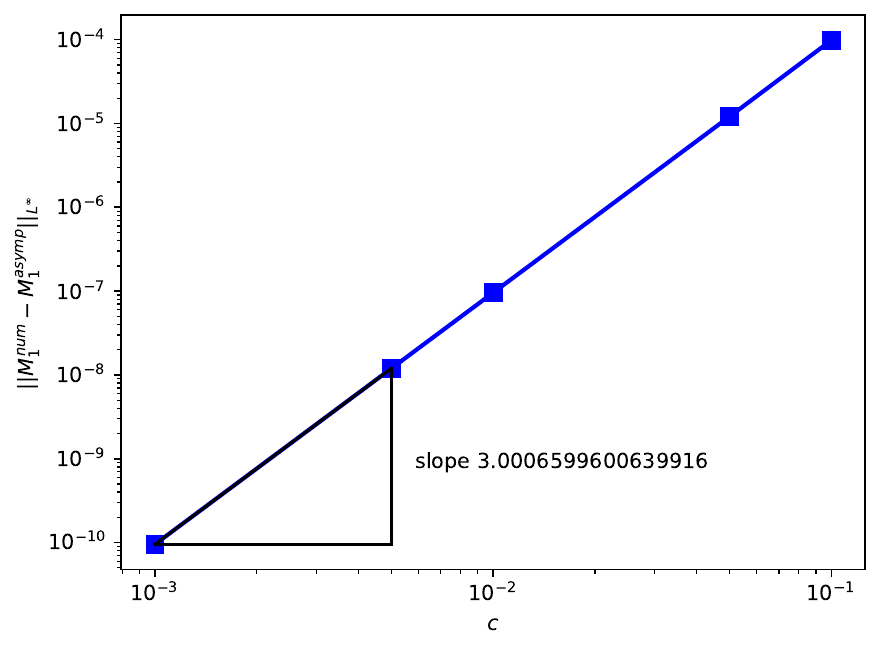}\\
\end{subfigure}\hfill
\caption{Log-log plots of $\|Q_{11}^{num} - Q_{11}^{asymp}\|_{L^\infty}$ (top row) and $\|M_1^{num}-M_1^{asymp}\|_{L^\infty}$ (bottom row).
    Left: truncating asymptotic expansions \cref{eq:Q-exp} and \cref{eq:M-exp} at $c^0$. 
    Middle: truncating asymptotic expansions at $c^1$. 
    Right: truncating asymptotic expansions at $c^2$.}
\label{fig:c-asymp}
\end{figure}

\section{Order reconstruction system: homogeneous solutions and asymptotics as $c\to 0$ and $c\to \infty$}
From the convergence result for $l\to \infty$ in the main text, we have that $\left(\Qvec^{OR}, \Mvec^{OR}\right) \to \left(-y, 0, -y, 0 \right)$ as $l\to \infty$.
The analysis in the $l\to 0$ limit is more involved, which necessitates a computation of the homogeneous solutions (with $\xi=1$) or critical points of the OR bulk potential
\begin{equation}
    f^{OR}(Q_{11},M_1)=(Q^2_{11}-1)^2+\frac{1}{4}(M^2_1-1)^2-cQ_{11}M_1^2,
    \label{eq:OR_homogeneous_energy}
\end{equation}
leading to the following algebraic equations:
\begin{subequations}\label{eq:OR_homogeneous_system}
    \begin{align}
        & 4 Q_{11}(Q_{11}^2 - 1) - cM_1^2=0 \label{eq:Q11_homogeneous_eqt},\\
        & M_1 (M_1^2 - 1) - 2 c Q_{11} M_1=0 \label{eq:M1_homogeneous_eqt}.
\end{align}
\end{subequations}
From \cref{eq:OR_homogeneous_system}, we have either $M_1=0$ and $Q_{11}=\pm1,0$ for any real-valued $c$, or $M_1^2=2cQ_{11}+1$.
Following the approach in \cref{sec:homo-full}, the non-trivial solutions of \cref{eq:OR_homogeneous_system} are
\begin{subequations}
\begin{align}
    & Q_{11}= \Theta_1 + \Theta_2,
    \label{eq:Q11_1}\\
    & Q_{11}= \omega_2 \Theta_1 +\omega_3 \Theta_2,
    \label{eq:Q11_2}\\
    & Q_{11}= \omega_3 \Theta_1 + \omega_2 \Theta_2,
    \label{eq:Q11_3}
    \end{align}
\end{subequations}
and the corresponding values of $M_1$ are
\begin{subequations}
\begin{align}
    & M_1=\pm
    \sqrt{2c\left( \Theta_1 + \Theta_2 \right) +1},
    \label{eq:M1_1}\\
    & M_1=\pm
    \sqrt{2c\left( \omega_2 \Theta_1 + \omega_3 \Theta_2 \right) + 1},
    \label{eq:M1_2}\\
    & M_1=\pm
    \sqrt{2c\left( \omega_3 \Theta_1 + \omega_2 \Theta_2 \right) +1}.
    \label{eq:M1_3}
\end{align}
\end{subequations}
Again, $\omega_1=1$, $\omega_2=\frac{-1+\sqrt{3}\mi}{2}$ and $\omega_3=\frac{-1-\sqrt{3}\mi}{2}$ are the cube roots of unity.
Therefore, the critical points of the OR bulk potential are (\cref{eq:Q11_1},\cref{eq:M1_1}), (\cref{eq:Q11_2},\cref{eq:M1_2}) and (\cref{eq:Q11_3},\cref{eq:M1_3}).
Since $27c^2< 64\left(1+\frac{c^2}{2}\right)^3$ for all $c\geq 0$,  \cref{eq:Q11_1}-\cref{eq:Q11_3} are real (by De Moivre's formula) and \cref{eq:M1_1}, \cref{eq:M1_3} are real for all positive $c$, whilst \cref{eq:M1_2} is  real for $c\leq\frac{1}{2}$.
Therefore, we ignore the pair (\cref{eq:Q11_2},\cref{eq:M1_2}) in what follows.
It is straightforward to verify from the same plot (see \cref{fig:energy_plots}) that the pair (\cref{eq:Q11_1},\cref{eq:M1_1}) is the global minimiser of the OR bulk potential \cref{eq:OR_homogeneous_energy}, for all values of $c$.

Next, we compute asymptotic expansions for (\cref{eq:Q11_1},\cref{eq:M1_1}), for small and large $c$.
We omit the details as they follow from \cref{sec:homo-full}.
For small $c$, one can check that
\begin{equation*}
    Q_{11}\approx1+\frac{c}{8},\;M_1^2 \approx 1 + 2c + \frac{c^2}{4}.
\end{equation*}
While for large $c$, we deduce that
\begin{equation*}
    Q_{11} \approx \left(\frac{\sqrt{2}}{4}\right)^\frac{1}{3}c,\; M_1^2\approx 1+\sqrt{2} c^2,
\end{equation*}
thus, $Q_{11}$ grows linearly in $c$ and $M_1^2$ grows quadratically in $c$, for $c\gg 1$.
\begin{remark}
    For $c=0$, there are nine solution pairs $(Q_{11},M_{1})$, given by \cref{eq:nine-sol-pairs}.
    In the limit $c\to 0$, we expect each of the solution pairs $(\cref{eq:Q11_1},\cref{eq:M1_1})$, $(\cref{eq:Q11_2},\cref{eq:M1_2})$ and $(\cref{eq:Q11_3},\cref{eq:M1_3})$ to reduce to a solution of the uncoupled problem.
    In fact, we can recover six solution pairs of the uncoupled system with $\sigma \neq 0$. 
    The solutions of the uncoupled system with $\sigma=0$, exist for all $c>0$, and are hence, unperturbed by the nemato-magnetic coupling for $c>0$.
\end{remark}
